\documentclass[leqno,11pt]{amsart}

\usepackage{palatino}
\usepackage[mathcal]{euler}

\usepackage{xcolor}

\definecolor{DarkBlue}{rgb}{0.2,0.2,0.4}

\usepackage{graphicx}
\usepackage{epstopdf,epsfig}
\usepackage{amsmath}
\usepackage{pst-node}
\usepackage{tikz-cd} 
\usepackage{comment}
\usepackage{xy}
\xyoption{all}
\usepackage{palatino}
\usepackage{tikz}
\usepackage{comment}
\usepackage[mathcal]{euler}
\usepackage{extarrows}
\usepackage{amsmath,amsthm,amsfonts,amssymb,latexsym,amscd,enumerate}
\usepackage{extarrows}

\usepackage{color}
\usepackage{slashed}

\theoremstyle{plain}
    \newtheorem{theorem}[equation]{Theorem}
    \newtheorem{lemma}[equation]{Lemma}
    \newtheorem{corollary}[equation]{Corollary}
    \newtheorem{proposition}[equation]{Proposition}

   \newtheorem*{theorem*}{Theorem}
 
 \theoremstyle{definition}
    \newtheorem{definition}[equation]{Definition}

    \newtheorem{remark}[equation]{Remark}
        \newtheorem*{remark*}{Remark}

\theoremstyle{remark}

\numberwithin{equation}{section}

\usepackage[colorlinks=true, linkcolor=DarkBlue, urlcolor= BrickRed, citecolor=DarkBlue]{hyperref}

\DeclareMathOperator{\ad}{ad}

\DeclareMathOperator{\End}{End}
\DeclareMathOperator{\Hom}{Hom}

\DeclareMathOperator{\Spin}{Spin}

 \DeclareMathOperator{\Ind}{Ind}

    \newcommand{\Dirac}{H_{\mathrm{Dirac}}}

\DeclareMathOperator{\Aut}{Aut}

\newcommand{\WBullet}{\,\,\,\bullet\,\,\,}

\begin{document}

\newcommand{\Spinc}{\Spin^c}
\newcommand{\Todo}{\textbf{To do}}

    \newcommand{\R}{\mathbb{R}}
    \newcommand{\C}{\mathbb{C}} 
    \newcommand{\N}{\mathbb{N}}
    \newcommand{\Z}{\mathbb{Z}} 
    \newcommand{\Q}{\mathbb{Q}}
    \newcommand{\bK}{\mathbb{K}}

\newcommand{\g}{\mathfrak{g}}
\newcommand{\h}{\mathfrak{h}}
\newcommand{\p}{\mathfrak{p}}
\newcommand{\kg}{\mathfrak{g}} 
\newcommand{\kt}{\mathfrak{t}}
\newcommand{\kA}{\mathfrak{A}}
\newcommand{\XX}{\mathfrak{X}}
\newcommand{\kh}{\mathfrak{h}} 
\newcommand{\kp}{\mathfrak{p}}
\newcommand{\kk}{\mathfrak{k}}
\newcommand{\kq}{\mathfrak{q}}
\newcommand{\ku}{\mathfrak{u}}
\newcommand{\kl}{\mathfrak{l}}
\newcommand{\kn}{\mathfrak{n}}
\newcommand{\ks}{\mathfrak{s}}
\newcommand{\ka}{\mathfrak{a}}
\newcommand{\km}{\mathfrak{m}}
\newcommand{\tr}{\mathrm{Trace}}

\newcommand{\cE}{\mathcal{E}}
\newcommand{\cA}{\mathcal{A}}
\newcommand{\calL}{\mathcal{L}}
\newcommand{\calH}{\mathcal{H}}
\newcommand{\cO}{\mathcal{O}}
\newcommand{\cB}{\mathcal{B}}
\newcommand{\cK}{\mathcal{K}}
\newcommand{\cP}{\mathcal{P}}
\newcommand{\calD}{\mathcal{D}}
\newcommand{\cF}{\mathcal{F}}
\newcommand{\cX}{\mathcal{X}}
\newcommand{\calM}{\mathcal{M}}
\newcommand{\calS}{\mathcal{S}}
\newcommand{\cU}{\mathcal{U}}

\newcommand{\Hilbert}{\mathcal{H}}

\newcommand{\Sj}{ \sum_{j = 1}^{\dim G}}
\newcommand{\Sk}{ \sum_{k = 1}^{\dim M}}
\newcommand{\ii}{\sqrt{-1}}

\newcommand{\ddt}{\left. \frac{d}{dt}\right|_{t=0}}

\newcommand{\PM}{P}
\newcommand{\DM}{D}
\newcommand{\LM}{L}
\newcommand{\vM}{v}

\newcommand{\Wedge}{\gamma}

\newcommand{\specialin}{\hspace{-1mm} \in \hspace{1mm} }

\newcommand{\beq}[1]{\begin{equation} \label{#1}}
\newcommand{\eeq}{\end{equation}}
\newcommand{\bspl}{\[ \begin{split}}
\newcommand{\espl}{\end{split} \]}

\newcommand{\Utilde}{\widetilde{U}}
\newcommand{\Xtilde}{\widetilde{X}}
\newcommand{\Dtilde}{\widetilde{D}}
\newcommand{\Etilde}{\widetilde{S}}
\newcommand{\wt}{\widetilde}
\newcommand{\Rep}{\mathrm{Rep}}
\newcommand{\ess}{\mathrm{ess}}

\newcommand{\Rhat}{\widehat{R}}

\title[On the Connes-Kasparov Isomorphism, II]{On the Connes-Kasparov Isomorphism, II:\\ The Vogan Classification of Essential Components in the Tempered Dual}

\author{Pierre Clare}
\author{Nigel Higson}
\author{Yanli Song}

\date{\today}
%\date{\today}

\begin{abstract}
This is the second of two papers dedicated to the computation of the reduced $C^*$-algebra of a connected, linear,
real reductive  group up to $C^*$-algebraic Morita equivalence, and the verification of the Connes-Kasparov conjecture in operator $K$-theory for these groups. In Part I we presented the  Morita equivalence and   the Connes-Kasparov morphism. In this part we shall compute the morphism using David Vogan's description  of the tempered dual.  In fact we shall go further by giving a complete representation-theor\-etic description and parametrization, in Vogan's terms, of the essential components of the tempered dual, which carry the $K$-theory of the tempered dual.
\end{abstract}

\maketitle
% \tableofcontents

\section{Introduction}
This is the second and concluding part of a work  whose objectives are  to  determine the reduced $C^*$-algebra of a connected, linear, real reductive group, up to Morita equivalence, compute its $K$-theory, and verify that the Connes-Kasparov index homomorphism in $K$-theory is an isomorphism.  

For further background on this problem we refer the reader to the introduction to Part I of this work \cite{ConnesKasparovPaper1}.
We showed there if $G$ is a real reductive group, then:

\begin{enumerate}[\rm (i)]

\item   The reduced $C^*$-algebra of $G$  may be described, up to C*-algebra isomorphism, in a way that neatly encapsulates the major results of tempered representation theory, as  developed by Harish-Chandra, Langlands, Knapp, Stein and others. 

\item Using the Knapp-Stein theory of the $R$-group and following Wassermann \cite{NoteWassermann}, the reduced group $C^*$-algebra may be described in a more streamlined way, up to Morita equivalence, and   its $K$-theory may be computed.  Only certain 
components of the tempered dual  contribute to $C^*$-algebra $K$-theory; these are the \emph{essential} components, and they may be characterized using the Knapp-Stein theory.

\item Contingent on a \emph{matching theorem} that pairs Clifford algebra data associated to a maximal compact subgroup of $G$ with the above essential components,  the $K$-theoretic indexes of the  Dirac operators associated to irreducible spin modules \cite[Def.~5.2]{ConnesKasparovPaper1} freely generate the $K$-theory   of the reduced group $C^*$-algebra.  This is the Connes-Kasparov isomorphism.
\end{enumerate}
In this paper we shall complete our  account of the Connes-Kasparov isomorphism in \cite{ConnesKasparovPaper1} by proving the matching theorem.\footnote{Actually, there are some other, lesser, issues outstanding from \cite{ConnesKasparovPaper1}.  We shall resolve those too,   in  Section~\ref{sec-more-on-dirac-cohomology}.} 
The main results    are:
\begin{enumerate}[\rm (i)]
\setcounter{enumi}{3}
\item A uniform construction of all the essential components of the tempered dual. 

\item A  parametrization of the essential components of the tempered dual in terms of dominant weights for a maximal compact subgroup of $G$.
\end{enumerate}
See Theorems~\ref{thm-classification-of-essential-discrete-data} and \ref{thm-essential-components-correspond-to-essential-data}.  Our proofs will  make very extensive use of   David Vogan's approach to the representation theory of $G$ \cite{Vogan79,Voganbook}, which involves minimal $K$-types and cohomological induction.  This is in contrast to most of the previous work on $C^*_r(G)$, which has used Harish-Chandra's approach, via discrete series representations and cuspidal parabolic induction.\footnote{A very noteworthy exception is Afgoustidis's construction of the Mackey bijection \cite{AfgoustidisMackeyBijection}, and the ensuing proof of the Connes-Kasparov conjecture \cite{AfgoustidisConnesKasparov} following \cite{Higson_Mackey}. Afgoustidis uses Vogan's theory of minimal $K$-types in an essential way.}

In brief, Vogan classifies the components of the tempered dual using  what we shall call sets of \emph{Vogan data}; we shall review this concept and the classification in Sections 2, 3 and 4.  In Sections 5 and 6 we shall describe how to list, and indeed construct,    the Vogan data that correspond to essential components.  
The parameters that we  use   turn out to be  \emph{exactly the same} as the parameters used to construct irreducible spin modules. This puts the proof of the matching theorem  within  reach. The proof is completed   in Section~\ref{sec-dirac-cohomology} by appealing to a result in \cite{DongHuang15} that relates cohomological induction, which is part of Vogan's classification method, to the Dirac operator.

\subsection*{Notes on Terminology} We shall generally follow the teminology and notation of Vogan's monograph \cite{Voganbook}. Thus throughout the paper, $G$ will be a  linear real reductive group in the sense of  \cite[Sec.~0.1]{Voganbook}, except that      we shall in addition assume that \emph{$G$ is connected as a real Lie group}, which Vogan does not. 

Throughout the  paper we shall work with a fixed  Cartan involution $\theta$ of $G$ and denote by  $K \subseteq G$ the (necessarily connected) maximal compact subgroup  that is fixed by $\theta$.  We shall also work with a fixed maximal torus $T^c\subseteq K$.

We shall denote  real Lie algebras of real groups by fraktur letters decorated with the subscript $0$, as in $ \mathfrak{k}_0\subseteq \mathfrak{g}_0$.  The same fraktur letters, but without the subscript $0$, will be used to denote complexified Lie algebras.  Other fraktur letters without the subscript $0$ will be used to denote other complex Lie algebras.

We shall write the Lie algebraic Cartan decomposition  as $\mathfrak{g}_0 = \mathfrak{k}_0 \oplus \mathfrak{s}_0$ (which is a slight departure from Vogan's notation)  and we shall fix, once and for all, a  nondegenerate, $G$-invariant symmetric bilinear form 
\begin{equation}
\label{eq-bilinear-form}
\langle\,\,,\,\rangle \colon \mathfrak{g}_0\times \mathfrak{g}_0 \longrightarrow \R
\end{equation}
that is positive-definite on $\mathfrak{s}_0$ and negative-definite on $\mathfrak{k}_0$.   We shall use various forms induced from \eqref{eq-bilinear-form},  most notably the associated positive-definite inner product on the space  $i \mathfrak{t}^{c*}_0 = \Hom_{\R} ( \mathfrak{t}^c_0, i \R)$.

Finally we shall denote by $\Delta(\mathfrak{k},\mathfrak{t}^c)$ the set of roots for $(\mathfrak{k},\mathfrak{t}^c)$. We shall fix, once and for all, a system of positive roots $\Delta^+(\mathfrak{k},\mathfrak{t}^c)$. The term \emph{dominant} (which means $\langle \lambda, \alpha\rangle \ge 0$ for all positive roots $\alpha$) will always mean \emph{with respect to this fixed system of positive roots}.  We shall use the standard notation $\rho(\Delta^+(\mathfrak{k},\mathfrak{t}^c))$ for the half-sum of positive roots, and we shall use similar notation in other contexts, where we shall always count roots with multiplicities. This is all as in \cite{Voganbook}.

%%%%%%%%%%%%%%%%%%%%%%%%%%%%%%%%%%%%%%%%%%%%%%%%%%%%%%%%%%%%%%%%%%%%%%%%%%%%%%%%%%%%

\section{Theta-Stable Parabolic Subalgebras} 
\label{sec-theta-stable-parabolics}

Vogan's monograph \cite{Voganbook} uses Zuckerman's method of  cohomological induction to construct the entire admissible dual of a real reductive group, and in particular the tempered dual.  In this section and the next we shall quickly review some of the basic  concepts that are involved. 

\begin{definition}[{See \cite[Def.\,5.2.1]{Voganbook}}] A  \emph{$\theta$-stable parabolic subalgebra} of $\mathfrak{g}$ is any Lie subalgebra  $\mathfrak{q}\subseteq \mathfrak{g}$ defined by a weight $\lambda\in i \mathfrak{t}_0^{c*}$ using the formula 
\[
 \mathfrak{q} = \bigoplus  \bigl \{ \mathfrak{g}_\beta : \beta \in i \mathfrak{t}^{c*}_0\,\, \text{and} \,\, \langle \lambda , \beta \rangle  \ge 0 \bigr \} ,
\]
where $\mathfrak{g}_\beta\subseteq \g$ is the $\beta$-weight space for the adjoint action of $\mathfrak{t}^c _0$ on $\mathfrak{g}$.  We shall say that $\mathfrak{q} $ is \emph{defined} by $\lambda$.
\end{definition}

The subalgebra $\mathfrak{q}$ decomposes as semidirect product $\mathfrak{q} = \mathfrak{l} + \mathfrak{u}$,  where 
\begin{equation}
    \label{eq-def-of-l}
 \mathfrak{l} = \bigoplus  \bigl \{ \mathfrak{g}_\beta : \beta \in i \mathfrak{t}^{c*}_0\,\, \text{and} \,\, \langle \lambda , \beta \rangle  =  0 \bigr \} 
\end{equation}
and 
\begin{equation}
    \label{eq-def-of-u}
 \mathfrak{u} = \bigoplus  \bigl \{ \mathfrak{g}_\beta : \beta \in i \mathfrak{t}^{c*}_0\,\, \text{and} \,\, \langle \lambda , \beta \rangle  > 0 \bigr \} .
\end{equation}
One has $\mathfrak{l} = \mathfrak{q} \cap \overline{\mathfrak{q}}$, so that $\mathfrak{l}$ is the complexification of a real algebra $\mathfrak{l}_0$. 

Throughout the paper, given a $\theta$-stable parabolic subalgebra $\mathfrak{q}$, we shall follow Vogan and denote by $L$ the normalizer of $\mathfrak{q}$ in $G$.  This is a connected Lie subgroup of $G$ and a reductive group in its own right, stable under the Cartan involution of $G$. Its Lie algebra is $\mathfrak{l}_0$.
 
We shall now present the examples of $\theta$-stable parabolic subalgebras that will be of  concern to us in this paper. 

\begin{proposition}
\label{prop-strictly-dominant-weight}
If the parabolic subalgebra $\mathfrak{q} \subseteq \mathfrak{g}$ is defined by a strictly dominant weight $\lambda\in i \mathfrak{t}^{c*}_0$, meaning that $\langle \lambda, \alpha\rangle > 0$ for all $\alpha\in \Delta^+(\mathfrak{k},\mathfrak{t}^c)$, then there is an isomorphism of Lie algebras 
\begin{equation*} 
% \label{eq-decomposition-of-essential-l}
\mathfrak{l}_0 \cong \mathfrak{sl}(2,\R) \oplus \cdots \oplus \mathfrak{sl}(2,\R) \oplus \mathfrak{z}_0
\end{equation*}
with the following properties: 
\begin{enumerate}[\rm (i)]

\item The summand $\mathfrak{z}_0$ is the center of $\mathfrak{l}_0$, and the isomorphism is the identity on $\mathfrak{z}_0$.

\item The isomorphism is compatible with the Cartan involution on $\mathfrak{l}_0$ that is obtained by restriction from $\mathfrak{g}_0$ and the standard Cartan involutions on $\mathfrak{sl}(2,\R)$-summands.

\item The summands on the right are orthogonal for the nondegenerate bilinear form on $\mathfrak{l}_0$ that is obtained by restriction from $\mathfrak{g}_0$.

\end{enumerate} 
 \end{proposition}

\begin{remark}
The $\mathfrak{sl}(2,
\R)$-summands on the right-hand side correspond to simple ideals in $\mathfrak{l}_0$, so the isomorphism is canonical up to permutation of the $\mathfrak{sl}(2,\R)$ summands and automorphisms of $\mathfrak{sl}(2,\R)$ preserving the Cartan involution.
\end{remark}

\begin{remark}
If $\lambda$ is strictly dominant not just for $\Delta^+(\mathfrak{k},\mathfrak{t}^c)$ but indeed for some system  of positive roots for the action of $\mathfrak{t}^c$ on $\mathfrak{g}$, then $\mathfrak{l}_0$ is in fact abelian; that is,  it has  no $\mathfrak{sl}(2,\R)$-summands. This is in fact the generic case, and it is very easy to analyze from the point of view of this paper. But it is not quite general enough for our purposes.
\end{remark}

\begin{proof}[Proof of Proposition~\ref{prop-strictly-dominant-weight}]
The hypothesis that $\lambda$ is \emph{strictly} dominant and the definition of $\mathfrak{l}$ imply that $\mathfrak{l}\cap \mathfrak{k} = \mathfrak{t}^c$. It follows that if we decompose $\mathfrak{l}$ into weight spaces for the adjoint action of $\mathfrak{t}^c$, 
\[
\mathfrak{l} = \mathfrak{t}^c \oplus \bigoplus _{\beta} \mathfrak{l}_\beta,
\]
then $\mathfrak{l}_\beta\subseteq \mathfrak{s}$ for all $\beta$. Hence the Cartan involution  $\theta$
is $-1$ on each $\mathfrak{l}_\beta$. 

Let us examine the weight spaces $\mathfrak{l}_\beta$ for $\beta \ne 0$. 
Minor variations on the standard arguments for semisimple Lie algebras (as in for example \cite[Sec.~VI.2]{Serre}) show that 
\[
\dim (\mathfrak{l}_\beta ) = 
\dim (\mathfrak{l}_{-\beta} ) = 
\dim ([\mathfrak{l}_\beta,\mathfrak{l}_{-\beta}]  ) = 1.
\]
So $\mathfrak{l}_{\beta}$,  $\mathfrak{l}_{-\beta}$ and $[ \mathfrak{l}_{\beta}, \mathfrak{l}_{-\beta}]$ span a $3$-dimensional Lie algebra isomorphic to $\mathfrak{sl}(2,\C)$.  It is stable under both the Cartan involution and complex conjugation, which switches $\mathfrak{l}_\beta$ and $\mathfrak{l}_{-\beta}$. So the Lie algebra is isomorphic to the complexification of  $\mathfrak{sl}(2,\R)$ with its standard Cartan involution.

If $\beta_1\ne -\beta_2$, then $[\mathfrak{l}_{\beta_1}, \mathfrak{l}_{\beta_2}] =0 $ since 
\[
[\mathfrak{l}_{\beta_1}, \mathfrak{l}_{\beta_2}] \subseteq \mathfrak{k} \quad \text{and} \quad 
[\mathfrak{l}_{\beta_1}, \mathfrak{l}_{\beta_2}]
\subset \mathfrak{l}_{\beta_1+ \beta_2}.
\]
Moreover $[ \mathfrak{l}_{\beta_1}, \mathfrak{l}_{-\beta_1}]$ and $[ \mathfrak{l}_{\beta_2}, \mathfrak{l}_{-\beta_2}]$ are orthogonal subspaces of $\mathfrak{t}^c$.  The proposition follows from these observations.
\end{proof}

At the Lie group level, we obtain from the isomorphism  of Lie algebras in Proposition~\ref{prop-strictly-dominant-weight}  a morphism
\begin{equation}
    \label{eq-morphism-from-product-of-sl2-lie-groups}
SL(2,\R) \times \cdots \times SL(2,\R) \times Z(L)^0 
\longrightarrow L
\end{equation}
(the superscript denotes the connected component of the identity) that is surjective and a local isomorphism.  The kernel is central, as it is for any local isomorphism, and its projection onto the product of $SL(2,\R)$-factors is injective because $Z(L)\subseteq L$. So the kernel identifies  via this projection with a subgroup of the (finite) group generated by the matrices 
$ \left [ \begin{smallmatrix}
-1 & 0 \\ 0 & -1 
\end{smallmatrix}\right ]
$
in each $SL(2,\R)$-factor. The maximal torus $T^c$, which is a subgroup of $L$, is the image of the corresponding morphism 
\begin{equation}
    \label{eq-morphism-from-product-of-so(2)-lie-groups}
SO(2) \times \cdots \times SO(2) \times (T^c\cap Z(L)^0 )
\longrightarrow L
\end{equation}
in which the final factor is the maximal compact subgroup of $Z(L)^0$, and is therefore connected.
\section{Vogan Data} 
\label{sec-discrete-theta-stable-data}

The following definition  will be used in the next section to construct  representations of $G$ using a two-step process that goes  from characters of a Cartan subgroup $H{\subseteq}G$ to representations of an intermediate group $L$ by means of parabolic induction, and then  from these representations of $L$ to representations of $G$ by means of  cohomological induction.

\begin{definition}[{See \cite[Def.~6.5.1]{Voganbook}}]  
\label{def-theta-stable-data}
A triple $(\mathfrak{q}, H, \delta)$ is a set of  \emph{Vogan data} for $G$ (Vogan uses the term \emph{discrete, $\theta$-stable data}) if
\begin{enumerate}[\rm (i)]
    \item 
    $\mathfrak{q} = \mathfrak{l} + \mathfrak{u}$ is a $\theta$-stable  parabolic subalgebra of $\mathfrak{g}$. 
    
    \item 
    $H$ is a $\theta$-stable Cartan subgroup of $G$ (it is necessarily abelian under our assumptions on $G$, but not necessarily connected) and a subgroup of $L=N_G(\mathfrak{q})$.

  \item  $L$ is \emph{quasi-split} (see \cite[Def.~4.3.5]{Voganbook}) and   $H$ is a \emph{maximally split Cartan subgroup} of $L$.  Altogether, this means that if we write $H = T A$, where $T = H \cap K$ and $A=  \exp [\mathfrak{h}_0\cap  \mathfrak{s}_0]$, then $H=TA$ is the Levi factor of  a minimal (real) parabolic subgroup $P= TAN$ of $L$.   
  
    \item
    $\delta\colon T\to U(1) $ is a  \emph{fine} representation of    $T$   with respect to $L$, in the sense of  \cite[Definition 4.3.8]{Voganbook}.   This means that $\delta$ is trivial on the connected component of the identity in the intersection of $T$ with   the semisimple part of $L$.
    
    \item 
    If $\lambda^L \in i \mathfrak{t}^*_0 = \Hom_{\R} (\mathfrak{t}_0 , i \R )$ is the differential of $\delta$, and if
    \[
    \lambda^G = \lambda ^L + \rho (\mathfrak{u},\mathfrak{t}) ,
    \]
    then it is required that $
    \langle \lambda ^G , \alpha \rangle > 0
    $
    for all weights  $\alpha$ for the adjoint action of $\mathfrak{t}$ on $\mathfrak{u}$. (The positive-definite inner product here is on $i \mathfrak{t}^{*}_0$, and it   is again obtained from the bilinear form on $\mathfrak{g}$. In addition $\rho (\mathfrak{u},\mathfrak{t})$ is the half-sum of the weights of the action of $\mathfrak{t}$ on $\mathfrak{u}$, multiplicities included.)
\end{enumerate}
\end{definition} 
 
Examples of Vogan data may be constructed using the following development of the computation in Proposition~\ref{prop-strictly-dominant-weight}. They represent all the examples that we shall need to study in detail in this paper.

\begin{proposition}
 \label{prop-theta-stable-data-from-dominant-weight}
Let  $\kappa\in  i \mathfrak{t}_0^{c*}$ be a dominant weight, and let $\mathfrak{q} =\mathfrak{l} + \mathfrak{u} $  be the $\theta$-stable parabolic subalgebra defined by the strictly dominant weight 
\[
\lambda^G  = \kappa + \rho(\Delta^+(\mathfrak{k},\mathfrak{t}^c)) .
\]
From the isomorphism in Proposition~\textup{\ref{prop-strictly-dominant-weight}}, the nonzero weights for the action of $\mathfrak{t}^c $ on $\mathfrak{l}$ consist of a collection of pairs $\pm \beta$, one for each $\mathfrak{sl}(2,\R)$-summand in $\mathfrak{l}_0$.  Denote by  $\Delta^+(\mathfrak{l},\mathfrak{t}^c)$ be any set consisting of one member from each pair. 
Let $H=TA$ be a maximally split Cartan subgroup of $L=N_G (\mathfrak{q})$. If the weight  
\[
\mu = \kappa - \rho (  \mathfrak{s}\cap\mathfrak{u}, \mathfrak{t}^c) - \rho (\Delta^+(\mathfrak{l}, \mathfrak{t}^c))
\]
is analytically integral, and if $\delta$ is the restriction to $T$ of  the   character $\exp(\mu)$ of $T^c$ associated to $\mu$, then  $(\mathfrak{q},H,\delta)$ is a set of Vogan data for $G$.   Whether or not  $\mu$ is integral is independent of the choice of $\Delta^+(\mathfrak{l}, \mathfrak{t}^c)$, and, assuming that $\mu$ is integral, the character $\delta$ is independent of the choice of $\Delta^+(\mathfrak{l}, \mathfrak{t}^c)$, too. 
\end{proposition}

\begin{remark}
\label{rem-weights-from-u-and-l-give-positive-system-for-g}
Taken together, the weights in $\Delta (\mathfrak{u}, \mathfrak{t}^c)$ and $\Delta^{+} (\mathfrak{l}, \mathfrak{t}^c)$ form a system of positive roots for   $(\mathfrak{g},\mathfrak{t}^c)$   that includes $\Delta^+(\mathfrak{k},\mathfrak{t}^c)$.  Moreover  if $\Delta^+(\mathfrak{s},\mathfrak{t}^c)$ is the subset of  noncompact roots in $\Delta^+(\mathfrak{g},\mathfrak{t}^c)$, then 
\[
\rho(\Delta^+(\mathfrak{s},\mathfrak{t}^c))=
\rho (  \mathfrak{s}\cap\mathfrak{u}, \mathfrak{t}^c) + \rho (\Delta^+(\mathfrak{l}, \mathfrak{t}^c)).
\]
The weight $\lambda ^G $ is dominant for $\Delta^+(\mathfrak{g},\mathfrak{t}^c)$. Conversely,  all of the instances of the construction in this remark, using all possible $\Delta^{+} (\mathfrak{l}, \mathfrak{t}^c)$, yield all the systems of positive roots for $(\mathfrak{g},\mathfrak{t}^c)$ for which $\lambda ^G$ is dominant. 
\end{remark}

\begin{remark}
Before starting the proof we should clear up a  notational ambiguity.  In the statement of   Proposition~\ref{prop-theta-stable-data-from-dominant-weight}, $\lambda^G$ is defined as a weight in $\mathfrak{t}^{c*}_0$ in terms of the dominant weight  $\kappa\in  i \mathfrak{t}_0^{c*}$. Then a weight $\mu\in i \mathfrak{t}_0^{c,*}$ is defined in terms of $\lambda^G$ and a choice of positive positive roots, and finally a unitary character $\delta$ of $T$ is defined in terms of $\mu$.  But in Definition~\ref{def-theta-stable-data}, a weight $\lambda^G\in i \mathfrak{t}_0^*$  is defined   in terms of the differential of $\delta$. So we are using the same symbol $\lambda ^G$ twice. 

Let us check that these two uses are related in the following way: if the weight $\lambda ^G\in i \mathfrak{t}_0^*$ of Definition~\ref{def-theta-stable-data} is extended by $0$ on the orthogonal complement of $\mathfrak{t}_0$ in $\mathfrak{t}^c_0$, then one obtains the weight $\lambda^G\in i \mathfrak{t}_0^{c*}$ in the statement of  Proposition~\ref{prop-theta-stable-data-from-dominant-weight}.

First, if $\lambda^G \in  i \mathfrak{t}_0^{c*}$ is the weight from  Proposition~\ref{prop-theta-stable-data-from-dominant-weight}, then $\lambda^G \vert _{\mathfrak{t}^\perp} =0$.   To see this, observe that in terms of the direct sum decomposition 
\begin{equation*} 
\mathfrak{l}_0 \cong \mathfrak{sl}(2,\R) \oplus \cdots \oplus \mathfrak{sl}(2,\R) \oplus \mathfrak{z}_0
\end{equation*}
in 
Proposition~\ref{prop-strictly-dominant-weight}, no matter how the maximally split Cartan subgroup $H=TA$ of $L$ is chosen, the Lie algebra $\mathfrak{t}_0$ is the $\theta$-fixed part of $ \mathfrak{z}_0$.  As a result, in terms of the same direct sum decompostion, the orthogonal complement of $\mathfrak{t}_0$ in $\mathfrak{t}^c$ is 
\begin{equation}
\label{eq-dir-sum-for-t-perp}
\mathfrak{t}_0^\perp \cong \mathfrak{so}(2)    \oplus \cdots \oplus  \mathfrak{so}(2)  .
\end{equation}
Let  $\beta\in i \mathfrak{t}^{c*}_0$ be a nonzero weight for the action of $\mathfrak{t}^c$ on (the complexification of) one of the $\mathfrak{sl}(2,\R)$-summands of $\mathfrak{l}_0$ in the direct sum decomposition. The restriction of $\beta$ to the orthogonal complement of the corresponding $\mathfrak{so}(2)$-summand  in $\mathfrak{t}_0^\perp$ is zero.  Hence
\[
\langle \lambda^G, \beta \rangle 
=
\langle \lambda^G\vert_{\mathfrak{so}(2)} , \beta \vert_{\mathfrak{so}(2)}  \rangle_{\mathfrak{so}(2)} .
\]
But according to the  definition of 
$\mathfrak{l}$ in \eqref{eq-def-of-l},   $\langle \lambda^G, \beta \rangle =0$, and so  it follows that $\lambda^G \vert_{ \mathfrak{so}(2) } = 0$ for each $\mathfrak{so}(2)$-summand in \eqref{eq-dir-sum-for-t-perp}. Hence $\lambda ^G \vert_{\mathfrak{t}^\perp} = 0$.

Second, it follows from the definition of  $\lambda^G
\in  i \mathfrak{t}_0^{c*}$ in the statement of the proposition that
\[
\lambda^G 
     =  \kappa  + \rho( \mathfrak{k}\cap \mathfrak{u},\mathfrak{t}^c)    = \kappa     - \rho ( \mathfrak{s}\cap \mathfrak{u}, \mathfrak{t}^c) 
      + \rho (\mathfrak{u},\mathfrak{t}^c),
\]
where $\rho( \mathfrak{k}\cap \mathfrak{u},\mathfrak{t}^c) $, etc, denote the half-sum of the weights of $\mathfrak{t}^c$ in $\mathfrak{k}\cap\mathfrak{u}$, etc.
 Meanwhile
\[
\lambda ^L = \mu \vert _{\mathfrak{t}} =  \kappa\vert _{\mathfrak{t}} - \rho (  \mathfrak{s}\cap\mathfrak{u}, \mathfrak{t}^c)\vert _{\mathfrak{t}} - \rho (\Delta^+(\mathfrak{l}, \mathfrak{t}^c))\vert _{\mathfrak{t}} ,
\]
and therefore 
\[
\lambda^L= \kappa\vert _{\mathfrak{t}} - \rho (  \mathfrak{s}\cap\mathfrak{u}, \mathfrak{t}^c)\vert _{\mathfrak{t}},
\]
since $\rho (\Delta^+(\mathfrak{l}, \mathfrak{t}^c))\vert _{\mathfrak{t}}=0$. This gives
\[
\lambda^G \vert_{\mathfrak{t}}= 
 \kappa  \vert_{\mathfrak{t}}  - \rho ( \mathfrak{s}\cap \mathfrak{u}, \mathfrak{t}^c) \vert_{\mathfrak{t}} + \rho (\mathfrak{u},\mathfrak{t}^c) \vert_{\mathfrak{t}}
    = \lambda ^L +  \rho (\mathfrak{u},\mathfrak{t}),
\]
as required.
\end{remark} 

In a somewhat similar vein, we shall use the following fact in the  proof of Proposition~\ref{prop-theta-stable-data-from-dominant-weight}:

\begin{lemma}
\label{lem-restriction-to-t-from-tc}
Let $\mathfrak{q} = \mathfrak{l} + \mathfrak{u}$ be the $\theta$-stable parabolic subalgebra defined by a strictly dominant weight in $i\mathfrak{t}^{c*}_0$.  If $H=TA$ is  a maximally split Cartan subgroup of $L=N_G (\mathfrak{q})$, then 
$
 \rho (\mathfrak{u},\mathfrak{t}^c)\vert_{\mathfrak{t}_0^\perp} = 0
$.
\end{lemma}

\begin{proof} 
The orthogonal complement of $\mathfrak{t}_0$ in $\mathfrak{t}_0^c$ is spanned by the images of the one-dimen\-sional Lie algebras $\mathfrak{so}(2)$ in the $\mathfrak{sl}(2,\R)$-summands of $\mathfrak{l}_0$, so it suffices to show that  $\rho (\mathfrak{u},\mathfrak{t}^c)$  restricts to zero on each $\mathfrak{so}(2)$. But the ideal $\mathfrak{u}$ is a representation, under the adjoint action, not just of $\mathfrak{so}(2)$ but of $\mathfrak{sl}(2,\R)$, and the sum of the weights of any finite-dimensional representation of $\mathfrak{sl}(2,\R)$ is zero.
\end{proof}

 \begin{proof}[Proof of Proposition~\ref{prop-theta-stable-data-from-dominant-weight}]
 The group $A \subseteq H$ is generated by images under the Lie group morphism  \eqref{eq-morphism-from-product-of-sl2-lie-groups} of $A$-subgroups in each of the $SL(2,\R)$-factors that map to $L$, along with the split part of the center of $L$.  Because of this, the group $T\subseteq H $  is generated by the torus $T^c \cap Z(L)^0$ and the images $m\in L$ of the matrices  
$ \left [ \begin{smallmatrix} -1 & 0 \\ 0 & -1 \end{smallmatrix}\right ]$ in each $SL(2,\R)$ factors in \eqref{eq-morphism-from-product-of-sl2-lie-groups}. The character $\delta $ is therefore fine, because the intersection of $T$ with the semisimple part of $L$ is the finite group generated by the elements $m$ alone.
 
Next, let us show that  if $\lambda^L\in i \mathfrak{t}^*_0$ is the differential of $\delta$, and if $\lambda ^G = \lambda^L + \rho (\mathfrak{u},\mathfrak{t})$ (this is the version of $\lambda ^G$ that is a weight on $\mathfrak{t}$, as in Definition~\ref{def-theta-stable-data}),
then  $\langle \lambda^G, \alpha \rangle  > 0$ for every   $\alpha\in i \mathfrak{t}^*$ that belongs to the set $\Delta (\mathfrak{u},\mathfrak{t})$ of weights for   the adjoint action of $\mathfrak{t}$ on $\mathfrak{u}$.
Given $\alpha\in \Delta(\mathfrak{u}, \mathfrak{t})$, there exists $\gamma\in \Delta (\mathfrak{u},\mathfrak{t}^c)$ with $\gamma\vert _{\mathfrak{t}} = \alpha$. It follows from the definition \eqref{eq-def-of-u} of $\mathfrak{u}$ and from the fact that $\lambda ^G \vert_{\mathfrak{t}^\perp} = 0$ that
\[
\langle \lambda^G, \alpha \rangle
=
\langle \lambda^G, \gamma \rangle > 0 ,
\]
as required.  

The above proves that $(\mathfrak{q},H,\delta)$ is a set of Vogan data.
It is clear that the difference of any two choices of $\Delta^+(\mathfrak{l},\mathfrak{t}^c)$ is analytically integral, so it remains to show that $\delta$ is independent of the choice of $\Delta^+(\mathfrak{l},\mathfrak{t}^c)$. 

 We shall show that $\delta(m) = -1$ for all of   the images in $L$ of the  elements $\left [\begin{smallmatrix}
-1 & 0 \\ 0 & -1
\end{smallmatrix}\right ]$ in the $ SL(2,\R)$-factors in \eqref{eq-morphism-from-product-of-sl2-lie-groups}, assuming of course that the weight 
\[
\mu = \kappa - \rho (  \mathfrak{s}\cap\mathfrak{u}, \mathfrak{t}^c) - \rho (\Delta^+(\mathfrak{l}, \mathfrak{t}^c))
\]
is an integral weight, so that $\delta$  is defined.  This will suffice, since $\delta$ is determined by its differential on the  torus $T^c \cap Z(L)^0$, and the differential is independent of the choice of $\Delta^+(\mathfrak{l},\mathfrak{t}^c)$.

Denote by $\pm \beta_j$ the two nonzero weights  for the action of $\mathfrak{t}^c_0$ on the $j$'th $\mathfrak{sl}(2)$-summand in $\mathfrak{l}$,  and arrange the signs so that
\[
\rho (\Delta^+ (\mathfrak{l}, \mathfrak{t}^c )) = \tfrac 12 \bigl ( \beta_1 + \cdots + \beta _N\bigr ).
\]
It follows from the definition \eqref{eq-def-of-l} of $\mathfrak{l}$ that 
\[
\langle \beta_j  , \kappa  \rangle  + \langle \beta_j ,  \rho (\mathfrak{k}\cap \mathfrak{u}, \mathfrak{t}^c)\rangle =
\langle \beta_j  , \kappa  \rangle  + \langle \beta_j ,  \rho (\Delta^+(\mathfrak{k}, \mathfrak{t}^c))\rangle 
  = \langle \beta_j , \lambda ^G \rangle  =  0
\]
(all of the inner products  above and below are taken in $i \mathfrak{t}^{c*}_0$) and so 
\[
\begin{aligned}
\langle \beta_j  ,  \kappa  \rangle -   \langle \beta_j ,  \rho (\mathfrak{s} \cap \mathfrak{u}, \mathfrak{t}^c)\rangle 
	 & = 	-  \langle \beta_j,  \rho ( \mathfrak{k}\cap \mathfrak{u}, \mathfrak{t}^c) \rangle
			- \langle \beta_j  ,  \rho (\mathfrak{s} \cap\mathfrak{u}, \mathfrak{t}^c)\rangle \\
	& = 
	- \langle \beta_j  ,  \rho (\mathfrak{u}, \mathfrak{t}^c)\rangle.
\end{aligned}
\]
But this is zero, since by Lemma~\ref{lem-restriction-to-t-from-tc} $\rho (\mathfrak{u},\mathfrak{t}^c)$ is supported on $\mathfrak{t}_0$, while $\beta_j$ is supported on a single $\mathfrak{so}(2)$-summand in $\mathfrak{t}_0^\perp$.
As a result,
\[
 \langle \beta_j  ,  \mu \rangle
	 = 
			- \langle \beta_j  , \rho (\Delta^{+} (\mathfrak{l}, \mathfrak{t}^c))\rangle 
	 =  - \tfrac 12 \langle \beta_j,\beta_j\rangle .
\]
It follows that the restriction of $\mu$ to the $j$'th  $\mathfrak{so}(2)$-summand is equal to $-\frac 12 \beta_j$. 
Now if $\beta$ is either of the two nonzero weights for the action of $\mathfrak{so}(2)$ on $\mathfrak{sl}(2,\R)$, then 
\[
\beta \Bigl (  \left [ \begin{smallmatrix}
0 & 1\\
-1 & 0 
\end{smallmatrix} \right ]   \Bigr ) = \pm 2 i,
\]
and so if $\phi\colon \mathfrak{sl}(2,\R)\to \mathfrak{l}_0$ is the inclusion of the $j$'th $\mathfrak{sl}(2,\R)$-summand, then 
\[
\mu \Bigl ( \phi \Bigl ( \left [ \begin{smallmatrix}
0 & 1\\
-1 & 0 
\end{smallmatrix} \right ]  \Bigr ) \Bigr ) = \pm  i.
\]
Using the fact that $\delta$ is the restriction to $T$ of the character $\exp (\mu)$ of $T^c$, we therefore find that 
  if $\Phi$ is the  Lie group morphism corresponding to $\phi$, then
\begin{equation*}
\begin{aligned}
 \delta (m ) 
    & =\exp(\mu)\Bigl (     \Phi \Bigl (  \left [ \begin{smallmatrix}
-1 & 0\\
0 & -1 
\end{smallmatrix} \right ]   \Bigr ) \Bigr) 
\\ 
    & =\exp(\mu)\Bigl ( \Phi \Bigl (\exp \Bigl ( \pi\cdot  \left [ \begin{smallmatrix}
0 & 1\\
-1 & 0 
\end{smallmatrix} \right ]   \Bigr )\Bigr ) \Bigr ) 
\\
    &   = \exp(\mu)\Bigl ( \exp  \Bigl ( \pi \cdot \phi \Bigl (  \left [ \begin{smallmatrix}
0 & 1\\
-1 & 0 
\end{smallmatrix} \right ]   \Bigr )\Bigr ) \Bigr ) 
\\
    & = \exp \Bigl (\pi \cdot \mu\Bigl (  \phi \Bigl  (\left [ \begin{smallmatrix}
0 & 1\\
-1 & 0 
\end{smallmatrix} \right ]  \Bigr )\Bigr ) \Bigr )
    = \exp (\pm \pi i ) =  -1 ,
\end{aligned}
\end{equation*}
 as required. 
 \end{proof}
 
To summarize, Proposition~\ref{prop-theta-stable-data-from-dominant-weight} associates to each dominant weight $\mu\in i \mathfrak{t}^{c*}_0$ a unique set of Vogan data, as long as $\mu$ satisfies a certain integrality condition.  We are now going to characterize the sets of Vogan data that are obtained in this way.

 \begin{definition}
 \label{def-essential-vogan-data}
A set $(\mathfrak{q},H,\delta)$ of Vogan data is  \emph{essential} if 
\begin{enumerate}[\rm (i)]

\item  There is an isomorphism 
\begin{equation*} 
% \label{eq-decomposition-of-essential-l}
\mathfrak{l}_0 \cong \mathfrak{sl}(2,\R) \oplus \cdots \oplus \mathfrak{sl}(2,\R) \oplus \mathfrak{z}_0
\end{equation*}
with the properties listed in Proposition~\ref{prop-strictly-dominant-weight}. 
\item 
The character  $\delta$ takes the value $-1$ on the images of each of the elements $\left [\begin{smallmatrix}
-1 & 0 \\ 0 & -1
\end{smallmatrix}\right ]\in SL(2,\R)$ under the associated  morphism of Lie groups
\[
SL(2,\R) \times \cdots \times SL(2,\R) \times Z(L)^0 
\longrightarrow L 
\]
in \eqref{eq-morphism-from-product-of-sl2-lie-groups}.
\end{enumerate}
 \end{definition} 
 
 \begin{remark}
The set of images in (ii) does not depend on the choice of isomorphism in (i).
 \end{remark}
 
 The proof that we have just completed shows that:

 \begin{proposition}
 \label{prop-essential-from-dominant-weight}
  Each of the sets of Vogan data provided by Proposition~\textup{\ref{prop-theta-stable-data-from-dominant-weight}} is essential. \qed
 \end{proposition}

\begin{theorem} 
\label{thm-classification-of-essential-discrete-data}
The construction in Proposition~\textup{\ref{prop-theta-stable-data-from-dominant-weight}} 
determines a bijection from the  dominant weights $\kappa \in i \mathfrak{t}^{c*}_0$  for which the weight 
\[
\mu = \kappa - \rho ( \mathfrak{s} \cap \mathfrak{u},\mathfrak{t}^c) - \rho(\Delta^+(\mathfrak{l},\mathfrak{t}^c))
\]
 is analytically integral to  the  $K$-conjugacy classes of sets of essential Vogan data.
\end{theorem}

Once again, we need a preliminary computation:

\begin{lemma}
\label{lem-fine-k-type-for-special-l}
Let $(\mathfrak{q},H ,\delta)$ be a set of essential Vogan data, and write $H=TA$, as usual, so that $\delta$ is a character of $T$.  For each positive system $\Delta^{+}(\mathfrak{l}, \mathfrak{t}^c)$
there is a unique analytically integral weight $\mu^L\in i \mathfrak{t}^{c*}_0$  for which 
\begin{enumerate}[\rm (i)]
    \item the corresponding global character $\exp(\mu^L)$ of $T^c$ restricts to $\delta$ on $T$, and 
    \item if $\beta\in \Delta^{+}(\mathfrak{l}, \mathfrak{t}^c)$ is the positive root associated to a given $\mathfrak{sl}(2,\R)$ summand in Definition~\textup{\ref{def-essential-vogan-data}}, then the restriction of $\mu^L$ to the corresponding $\mathfrak{so}(2)$-summand of $\mathfrak{t}^c$ is  
    \[
    \mu^L\vert_{\mathfrak{so}(2)} = - \tfrac 12 \beta\vert_{\mathfrak{so}(2)}.
    \]
\end{enumerate}
\end{lemma}

\begin{remark}
    \label{rem-fine-weigths-of-tc-1}
In the essential case that we are considering in the lemma,  conditions (i) and (ii) above characterize what Vogan calls the \emph{fine} representations $\exp(\mu^L)$ of $T^c$ with respect to $L$ that restrict to the character $\delta$ on $T$ \cite[Def.~4.3.9]{Voganbook}.  The set of all such is denoted $A(\mathfrak{q},\delta)$ by Vogan \cite[Def.~4.3.15]{Voganbook}, and the lemma shows that it has $2^N$ elements, where $N$ is the number of $\mathfrak{sl}(2,\R)$ factors in $\mathfrak{l}_0$.
\end{remark} 

\begin{remark}
    \label{rem-fine-weigths-of-tc-2}
Later on it will  be more appropriate to change a sign so that 
\[
    \mu\vert_{\mathfrak{so}(2)} = \tfrac 12 \beta\vert_{\mathfrak{so}(2)} ,
\]
which of course we can do by working with the opposite system of positive roots for $(\mathfrak{l},\mathfrak{t}^c)$.  In this case 
\[
\mu^L = \kappa^L + \rho (\Delta ^+(\mathfrak{l},\mathfrak{t}^c)),
\]
where $\kappa^L\in i \mathfrak{t}_{0}^{c,*}$ is the differential of $\delta$, extended by zero from $\mathfrak{t}_0$ to $\mathfrak{t}^c$.
\end{remark}

\begin{proof}[Proof of the Lemma]
The lemma  is a special case of one of the most important technical results in Vogan's monograph, \cite[Thm.~4.3.16]{Voganbook}.  
But the special case is easy to handle directly, as follows.

We shall use the  surjective morphism of Lie groups \eqref{eq-morphism-from-product-of-sl2-lie-groups}. As indicated in \eqref{eq-morphism-from-product-of-so(2)-lie-groups}, the torus $T^c$ is generated by the image of $
SO(2)\times\cdots \times SO(2)$
and the compact part of the center of $L$. Because of this we need only  define a unitary character on $SO(2)\times\cdots \times SO(2)$, namely a product of generating characters on each factor, and a unitary character on the compact part of the center, namely the restriction of $\delta$, and then check that the product factors through 
\[
\Bigl ( SO(2)\times \cdots \times SO(2)\Bigr ) \times \Bigl ( T^c \cap Z(L)\Bigr ) \longrightarrow T^c 
\]
using the   property of $\delta$ in the definition of essential Vogan data.  This is straightforward.
\end{proof}

\begin{proof}[Proof of Theorem~\ref{thm-classification-of-essential-discrete-data}]
Let $(\mathfrak{q},H,\delta)$ be a set  of essential Vogan data. After conjugating by an element of $K$ we may, and shall, assume that $\mathfrak{u}\cap \mathfrak{k}$ is the direct sum of the already-fixed positive weight spaces in $\mathfrak{k}$.  

Let $\lambda^L$ be the differential of $\delta$ and  define the weight   $\lambda^G \in i\mathfrak{t}^*_0$ by
\[
\lambda^G = \lambda^L + \rho (\mathfrak{u},\mathfrak{t} ) .
\]
If we extend $\lambda^G$ to a weight in $i \mathfrak{t}^{c*}_0$ by defining it to be zero on the orthogonal complement of $\mathfrak{t}_0\subseteq \mathfrak{t}^c$, then it follows from part (v) of Definition~\ref{def-theta-stable-data} that the extension satisfies
\[
\langle \lambda ^G , \alpha \rangle  > 0 \quad \forall \alpha \in \Delta (\mathfrak {u}, \mathfrak{t^c}) ;
\]
compare the proof of Proposition~\ref{prop-theta-stable-data-from-dominant-weight}. In addition, 
\[
\langle \lambda ^G , \beta \rangle  = 0 \qquad  \forall\, \beta \in \Delta (\mathfrak{l}, \mathfrak{t}^c) , 
\]
since $\lambda^G$ vanishes on $\mathfrak{t}_0^\perp$ while $\beta$ vanishes on $\mathfrak{t}_0$.  So   $\lambda ^G $  defines $\mathfrak{q}$.

Because the maximal compact subgroup of $L$ is the torus $T^c$, the maximally split Cartan subgroup $H =TA$ of $L$ is unique up to conjugacy by an element of $T^c$, and $T$ is uniquely determined, and a subgroup of $T^c$.

Choose an integral weight $\mu^L$ of $\mathfrak{t}^c$  as in Lemma~\ref{lem-fine-k-type-for-special-l}, so that 
\[
\mu^L + \rho (\Delta^+(\mathfrak{l},\mathfrak{t}^c)) = \lambda ^L ;
\]
here we extend $\lambda ^L$ by zero on $\mathfrak{t}_0^\perp$. The unitary character of $T^c$ associated to $\mu^L$ restricts to $\delta$ on $T\subseteq T^c$.  Indeed the differential of this character is equal to $\lambda ^L$ on $\mathfrak{t}$, so the character agrees with $\delta$ on the connected component of the identity in $T$. But in addition both the character and $\delta$ are equal to $-1$ on the elements $m$ that generate the component group of $T$.

Next, define a weight  $\kappa\in i\mathfrak{t}^{c*}_0$   by 
\[
\mu = \kappa - \rho ( \mathfrak{s} \cap \mathfrak{u},\mathfrak{t}^c) - \rho(\Delta^+(\mathfrak{l},\mathfrak{t}^c))
\]
and observe that, thanks to Lemma~\ref{lem-restriction-to-t-from-tc},
\[
\kappa + \rho (\Delta^+(\mathfrak{k},\mathfrak{t}^c)) 
    =
\mu^L+ \rho (\Delta ^+(\mathfrak{l},\mathfrak{t}^c))
 + \rho (\mathfrak{u},\mathfrak{t}^c) 
    = \lambda ^L + \rho (\mathfrak{u}, \mathfrak{t}^c) 
      = \lambda ^G,
\]
and so $\kappa +\rho (\Delta^+(\mathfrak{k},\mathfrak{t}^c)) $ defines $\mathfrak{q}$.   

It follows from the above that $\kappa + \rho (\Delta^+(\mathfrak{k},\mathfrak{t}^c))$ is strictly dominant.  If $\kappa$ was integral, then it   would follow from this that $\kappa$ is dominant.  But $\kappa$ \emph{is} integral for a certain double cover $\widetilde K$ of $K$, namely the one defined in Section~\ref{sec-dirac-cohomology} below; see the discussion following Proposition~\ref{prop-genuine-reps-of-k-tilde-by-highest-weight}. So $\kappa$ is dominant whether or not it is integral.  

We have now shown that $(\mathfrak{q},H,\delta)$ is one of the sets of essential Vogan data  defined in Proposition~\ref{prop-essential-from-dominant-weight}, up to $K$-conjugacy.  Since there are no inner automorphisms of $K$ that globally preserve  both  $T^c$ and the system of positive roots $\Delta^+(\mathfrak{k},\mathfrak{t}^c)$, the map from weights to $K$-conjugacy classes of essential Vogan data is not only surjective, but injective too.
\end{proof}

\section{Components of the Tempered Dual}
\label{sec-components-of-tempered-dual}

We shall now review   Vogan's construction of the tempered dual\footnote{Vogan constructs the full admissible dual in \cite{Voganbook}; the fact that we are interested only in the tempered dual simplifies matters for us considerably.} of $G$,   and compare it with  Harish-Chandra's construction, which  we used in \cite{ConnesKasparovPaper1} to describe the reduced group $C^*$-algebra.

Vogan's construction is carried out in the context of $(\mathfrak{g},K)$-modules rather than that of unitary Hilbert space representations.  But we shall take advantage of the fact, due to Harish-Chandra, that every unitary admissible $(\mathfrak{g},K)$-module may be completed to a unitary representation of $G$, and usually make no distinction between the two contexts.  But whenever  it is helpful to do so we shall denote $(\mathfrak{g},K)$-modules using the letter $X$, and  Hilbert space representations using the letter $\Hilbert$.

The following definition formalizes the construction of representations from Vogan data.

\begin{definition}
\label{def-vogan-representations}
Let $(\mathfrak{q}, H, \delta)$ be a set of Vogan data for $G$, and let $L=N_G (\mathfrak{q})$. Let $T=MA$ as in Definition~\textup{\ref{def-theta-stable-data}}, and let $\varphi \in \mathfrak{a}_0^*$. 
\begin{enumerate}[\rm (i)]

\item We shall denote by  $X^L(\delta, i \varphi )$ the unitary principal series representation of $L$ that is obtained by unitary parabolic induction from  the character $\delta{\otimes} \exp (i \varphi)$ of $H$.  When $\varphi=0$ we shall abbreviate the notation to $X^L(\delta)$ or even $X(\delta)$.

\item We shall denote by  $X^G(\mathfrak{q}, H ,\delta, i \varphi)$ the representation of $G$ that is obtained by  cohomologically inducing $X^L(\delta, i \varphi)$ from $L$ to $G$  via the parabolic subalgebra $\mathfrak{q}$.  See \cite[Def.~6.5.2]{Voganbook}; this is the \emph{standard representation with parameters $(\mathfrak{q}, H ,\delta, i \varphi)$}.  Again, when $ \varphi=0$ we shall abbreviate the notation to $X^G(\mathfrak{q},H,\delta)$ or $X(\mathfrak{q},H,\delta)$.

\end{enumerate}
\end{definition}

\begin{remark}
\label{rem-vogan-notation}
Vogan studies the representations $X^G(\mathfrak{q}, H ,\delta, \nu)$
associated to all real-linear $\nu\colon \mathfrak{a}_0\to \C$.  Because our focus is on the tempered dual only, we are concerned only with imaginary-valued $\nu$, which we have written as $\nu = i \varphi$, as above.
\end{remark}

Up to equivalence, the representation $X^G(\mathfrak{q}, H ,\delta, i \varphi)$ depends only on the $K$-conjugacy class of $(\mathfrak{q},H,\delta, i \varphi)$. It is unitarizable,  tempered and a finite direct sum of irreducible representations (generically $X^G(\mathfrak{q}, H ,\delta, i \varphi)$ is  itself irreducible). This follows, for example, from Theorem~\ref{thm-vogan-zuckerman1} below. 

The representations $X^G(\mathfrak{q}, H ,\delta, i \varphi)$ fit into the Harish-Chandra picture of the tempered dual as follows.
According to Harish-Chandra, the components of the tempered dual of $G$ are parametrized by \emph{associate classes} of pairs $(P,\sigma)$ consisting of a parabolic subgroup $P$ of $G$ and an irreducible, square-integrable representation $\sigma$ of the compactly generated part $M$ in the Langlands decomposition $P=MAN$.  See for instance \cite{ConnesKasparovPaper1} for a summary that is tailored to our viewpoint.

Let  $[P,\sigma]$ be an associate class,  with $P=MAN$. We shall construct from it a set of Vogan data.

First, choose a Cartan subgroup  $T$ (centralizer of a maximal torus) in   $K\cap M$, and set 
$H = TA$.

The compact factor  $T$ is a compact Cartan subgroup of $M$, and the Vogan-Zuckerman method attaches to the discrete series representation  $\sigma$ of $M$ a character $\gamma$ of $T$ from which $\sigma$ may be obtained by cohomological induction using an appropriate $\theta$-stable parabolic subalgebra of $\mathfrak{m}$; see   \cite[Lem.~6.6.12]{Voganbook}. 

The character $\gamma$, in turn, determines a  $\theta$-stable  parabolic subalgebra $\mathfrak{q}\subseteq \mathfrak{g}$, as well as  a character $\delta$ of $T$, which is a $\rho$-shift of $\gamma$, and we obtain a set of Vogan data $(\mathfrak{q}, H, \delta)$; see \cite[Prop.~6.6.2]{Voganbook}.

The $K$-conjugacy class $[\mathfrak{q}, H, \delta]$ depends only on the associate class $[P,\sigma]$, and the representations of $G$ associated to the two sets of data are related as follows:

\begin{theorem}[{\cite[Thm.~6.6.15]{Voganbook}}]
\label{thm-vogan-zuckerman1}
Suppose that the $K$-conjugacy class of Vogan data $[\mathfrak{q},H,\delta]$ corresponds to the associate class  $[P,\sigma]$, as above. For every $\varphi \in \mathfrak{a}_0^*$, the unitary  $(P,\sigma)$-principal series representation $\pi_{ \sigma, \varphi}$  on $ \Ind_P^G \Hilbert_\sigma{\otimes} \C_{i \varphi} $
is equivalent to the representation  $X^G(\mathfrak{q}, H ,\delta, i \varphi)$ described in Definition~\textup{\ref{def-vogan-representations}}.  \qed
\end{theorem}

Putting this together with   known facts about how associate classes organize the tempered dual, as summarized in \cite{ConnesKasparovPaper1} for instance, or alternatively working directly with Vogan data, we obtain: 

\begin{theorem}
\label{thm-vogan-zuckerman0}
Let $(\mathfrak{q}, H, \delta)$ be a set of Vogan data for $G$, let $H=TA$ as in Definition~\textup{\ref{def-theta-stable-data}}, and let $\varphi \in \mathfrak{a}_0^*$.  

\begin{enumerate}[\rm (i)]

\item Every tempered irreducible representation of $G$ arises    as one of the summands of some $X^G(\mathfrak{q}, H ,\delta, i \varphi)$, and it does so from a unique set of Vogan data, up to $K$-conjugacy.

\item As $\varphi$ varies, with $(\mathfrak{q}, H, \delta)$ fixed, the irreducible summands of  all the representations  $X^G(\mathfrak{q}, H ,\delta, i \varphi)$  belong to a single component of the tempered dual, and exhaust it.   \qed

\end{enumerate}
\end{theorem}

In short, the components of the tempered dual of $G$ are  param\-etrized by the $K$-conjugacy classes of sets of  Vogan data, and  this param\-etrization determines a bijection 
\begin{equation*}
    \label{eq-bijection-of-parameters}
\bigl \{ \text{\rm Associate classes $[P,\sigma]$}\,\bigr \} 
\stackrel \cong \longrightarrow 
\bigl \{ \, \text{\rm $K$-conjugacy classes $[\mathfrak{q},H,\delta]$} \,\bigr \} .
\end{equation*}

%%%%%%%%%%%%%%%%%%%%%%%%%%%%%%%%%%%%%%%%%%%%%%%
%%%%%%%%%%%%%%%%%%%%%%%%%%%%%%%%%%%%%%%%%%%%%%%

\section{The Knapp-Stein and Vogan Intertwining Groups}
\label{sec-intertwining-groups}

Let $[P,\sigma]$ be an associate class, as in the previous section. Harish-Chandra showed that the decomposition of the $(P,\sigma)$-principal series  representation $\pi_{\sigma,\varphi}$
into irreducible constituents is governed by the interwining group
\[
W_\sigma= \bigl \{w \in N_K(\ka_0)/Z_K(\ka_0) : \operatorname{Ad}_w ^* \sigma \simeq \sigma\, \bigr \} ,
\]
and  more specifically by the subgroup
\[
W_{\sigma,\varphi}  = \bigl \{w \in W_\sigma : \operatorname{Ad}_w ^* \varphi = \varphi  \, \bigr \} .
\]
Each element of $W_{\sigma,\varphi}  $   corresponds to  a unitary self-intertwiner of the representation $\pi_{\sigma,\varphi}$ on  $\Ind_P^G \Hilbert_\sigma{\otimes} \C_{i \varphi}$, and  together these unitaries span the intertwining algebra. See \cite[Thm.~38.1]{HC3}. 

In addition, Langlands proved that the unitary principal series representations  $\pi_{\sigma,\varphi}$ and $\pi_{\sigma',\varphi'}$ on  $\Ind_P^G \Hilbert_\sigma {\otimes} \C_{i \varphi}$
and $\Ind_{P'}^G \Hilbert_{\sigma'} {\otimes} \C_{i\varphi'}$ are equivalent if and only if the data $(A,\sigma, \varphi)$ and  $(A',\sigma', \varphi')$ are conjugate by an element of $K$, and that otherwise the two representations are disjoint. See \cite[pp.142~\&~149-151]{MR1011897} or \cite[Thm.~14.90]{Knapp1}.

Now let  $(\mathfrak{q},H,\delta)$ be a set of Vogan  data. As usual let  $L=N_G(\mathfrak{q})$ and let  $H=TA$. If we  set
\begin{equation*}
% \label{eq-def-of-w-for-l}
    W  = N_{K\cap L}(\ka_0)/Z_{K \cap L}(\ka_0) ,
\end{equation*}
then the  intertwining groups in the Vogan-Zuckerman theory are  
\begin{equation*}
%    \label{eq-def-of-w-delta-group}
W_\delta= \bigl \{w \in W  : \operatorname{Ad}_w ^* \delta = \delta \bigr \}
\end{equation*}
and 
\[
W_{\delta,\varphi}  = \bigl \{w \in W_\delta : \operatorname{Ad}_w ^* \varphi = \varphi  \, \bigr \} .
\]
See \cite[Defs.~4.4.1 and 4.4.9]{Voganbook}.  

\begin{theorem}[{\cite[Thm.~6.5.12]{Voganbook}}]
Let $(\mathfrak{q}',H',\delta')$ be another set of Vogan data and let $\varphi'\in {\mathfrak{a}'}_{\!\!0}^*$.  If the representations $X^G(\mathfrak{q}, H ,\delta, i \varphi)$ and $X^G(\mathfrak{q}', H' ,\delta', i \varphi')$ have equivalent subrepresentations, then $(\mathfrak{q},H,\delta,i \varphi)$ and  $(\mathfrak{q}',H',\delta',i \varphi')$ are conjugate by an element of $K$.
\end{theorem}

Let  $(\mathfrak{q},H,\delta)$ be the set of Vogan data that is constructed from $(P,\sigma)$ as in the discussion prior to the statement of Theorem~\ref{thm-vogan-zuckerman1}.  Since  both the Vogan and Knapp-Stein   intertwining groups   may be faithfully represented as groups of linear automorphisms of the same vector space $\mathfrak{a}_0$,  they may be compared with one another. In fact they are equal:

\begin{lemma}
\label{lem-w-sigma-equals-w-delta}
If  $(\kq, H, \delta)$ is the set of  Vogan data constructed from  the pair $(P, \sigma)$, as in the discussion prior to Theorem \textup{\ref{thm-vogan-zuckerman1}}, then $W_\sigma =  W_\delta$.
\end{lemma}

\begin{proof}
Suppose that $w$ is an element in $W_\delta$, acting on $\ka_0^*$ through $w(\varphi)=\operatorname{Ad}_w ^* \varphi$. Then the representations $X^G(\mathfrak{q}, H ,\delta, i\varphi)$ and $X^G(\mathfrak{q}, H ,\delta, iw(\varphi))$ are equivalent, for all $\varphi$, and therefore by Theorem~\ref{thm-vogan-zuckerman1} the representations $\pi_{\sigma,\varphi}$ and $\pi_{\sigma,w(\varphi)}$ are equivalent, for all $\varphi$.  So by Langlands' aforementioned result, for each $\varphi$ there is  an element $w'_\varphi \in W_\sigma$ such that $w'_\varphi(\varphi)  = w(\varphi)$.

Since $W_\sigma$ is finite, the elements $w'_\varphi$  cannot all be distinct from one another. Indeed there must exist a single $w'\in W_\sigma$ such that 
\[
w'(\varphi) = w'_\varphi (\varphi) = w(\varphi)
\]
for all $\varphi$ in a spanning set for $\mathfrak{a}_0^*$.  This means that $w'=w$, and we have proved that $W_\delta \subseteq W_\sigma$.  The argument can be reversed to prove the reverse inclusion.
\end{proof}
 
From now on we shall focus on the special case $\varphi =0$. The Knapp-Stein theory \cite{KS1,KS2} considerably refines Harish-Chandra's completeness result by decomposing $W_\sigma$ as a semidirect product 
\begin{equation}
\label{eq-knapp-stein-semidirect-product}
W_\sigma = W'_\sigma \rtimes R_\sigma ,
\end{equation}
in which intertwiners corresponding to the elements of $W'_\sigma$ all act as multiples of the identity on the Hilbert space $\Ind_P^G \Hilbert_\sigma$ of the representation $\pi_{\sigma,0}$, while the intertwiners corresponding to the elements of $R_\sigma$ act as a \emph{linearly independent} set of \emph{mutually commuting} intertwining operators on  $\Ind_P^G \Hilbert_\sigma$.  Moreover $R_\sigma$ is abelian and (after possibly adjusting each intertwiner by a multiplicative constant), it acts as a \emph{group} on $\Ind_P^G \Hilbert_\sigma$.   See \cite[Thm.~13.4]{KS2} and \cite[Main Thm.,~p.34]{Knapp_commut}, as well as \cite[Sec.~2]{ConnesKasparovPaper1} for a further discussion of these important issues.  It follows from these facts and Harish-Chandra's theorem that the intertwining algebra for the $(P,\sigma)$-principal series representation $\pi_{\sigma,0}$ is isomorphic to the group algebra of $R_\sigma$.   As a result,  the $(P,\sigma)$-principal series representation $\pi_{\sigma,0}$ decomposes as a direct sum of a finite set of mutually inequivalent irreducible representations that may be parametrized by the elements of the dual group $\widehat R_\sigma$.

We shall now present  some of the corresponding results in the Vogan-Zucker\-man theory.

\begin{definition}[{\cite[Def.~4.1.1]{Voganbook}}]
\label{def-of-delta-l}
Let  $\mathfrak{q}$ be a $\theta$-stable parabolic subgroup of $\mathfrak{g}$.  Let $L=N_G(\mathfrak{q})$, as usual, and let  $H=TA$ be a maximally split Cartan subgroup of $L$.   
\begin{enumerate}[\rm (i)]

    \item Denote by $\Delta(\mathfrak{l},\mathfrak{a}) {\subseteq} \mathfrak{a}_0^*$ the set of (non-zero) roots associated to the adjoint action of $\mathfrak{a}$ on $\mathfrak{l}$. These are the \emph{restricted roots} for $(\mathfrak{q},H,\delta)$. 
    
    \item Denote by $\overline{\Delta}(\mathfrak{l},\mathfrak{a}) {\subseteq} \mathfrak{a}_0^*$ set of those roots  $\alpha \in {\Delta}(\mathfrak{l},\mathfrak{a}) $ for which    $\frac{\alpha}{2} \notin {\Delta}(\mathfrak{l},\mathfrak{a}) $. These are the \emph{reduced roots} for $(\mathfrak{q},H,\delta)$. 
    
\end{enumerate}
\end{definition}

The following lemma  is a standard fact about  the set of restricted roots associated to the Iwasawa decomposition of any reductive group \cite[Thm.~6.57]{KnappBeyond}.

\begin{lemma} 
If  $(\mathfrak{q},H,\delta)$ is a set of Vogan data,  and if $H=TA$ and   $L=N_G(\mathfrak{q})$, as in Section~\ref{sec-components-of-tempered-dual}, then 
    \[
W(\overline \Delta ( \mathfrak{k},\mathfrak{l})) = N_{K\cap L} (\mathfrak{a}_0) / Z_{K\cap L} (\mathfrak{a}_0)
\]
if both groups are viewed as groups of linear  automorphisms of $\mathfrak{a}_0$. \qed
\end{lemma}

\begin{definition}[{\cite[Def.~4.3.6]{Voganbook}}]
\label{def good root}
A root  $\alpha \in \overline{\Delta}(\mathfrak{l},\mathfrak{a}) $  is said to be  \emph{real}  if it is the restriction to $\mathfrak{a}$ of a root for the adjoint  action   of $\mathfrak{h} = \ka \oplus \kt$ on $\kl$ that vanishes on $\mathfrak{t}$. Otherwise,   $\alpha$ is said to be \emph{complex}. 
\end{definition}

Associated to each real root $\alpha\in \overline{\Delta}(\mathfrak{l},\mathfrak{a}) $, there exists a morphism (\cite[Notation 4.3.6]{Voganbook})
\begin{equation}
    \label{eq-lie-algebra-morphism-phi-alpha}
\phi_\alpha \colon \mathfrak{sl}(2, \R) \to \kl_0  
\end{equation}
that is compatible with $\theta$ and the standard Cartan involution on $\mathfrak{sl}(2,\R)$, and that maps the strictly upper triangular matrices into the $\alpha$-root space in $\mathfrak{l}_0$. See \cite[Def.~4.3.6]{Voganbook}. Since the group   $L$ is linear,  the Lie algebra morphism $\phi_\alpha$ exponentiates to a Lie group  morphism 
\begin{equation}
    \label{eq-lie-group-morphism-phi-alpha}
\Phi_\alpha \colon SL(2, \R) \to L .
\end{equation}
Using this we put 
\begin{equation}
    \label{eq-element-m-alpha}
m_\alpha = \Phi_\alpha \left ( \left [ \begin{smallmatrix}
-1 & 0 \\
0 & -1 
\end{smallmatrix} \right ] \right )\in L. 
\end{equation}
This is in fact an element of $T$, and it is independent of the choice of $\phi_\alpha$ \cite[Lem.~4.3.7]{Voganbook}. 

Of course, in the case of an essential set of Vogan data, all these structures match those introduced in Section~\ref{sec-discrete-theta-stable-data}.

\begin{definition}[{\cite[Notation~4.3.6]{Voganbook}}]
\label{def-good-roots}
Let $(\kq, H, \delta)$ be a set of Vogan data. 
The set of \emph{good roots with respect to $\delta$}, denoted 
$\Delta_\delta (\mathfrak{l},\mathfrak{a}) \subseteq \overline{\Delta}(\mathfrak{l},\mathfrak{a}) $,
is the set of  those roots $\alpha\in \overline{\Delta}(\mathfrak{l},\mathfrak{a}) $ such that either
\begin{enumerate}[\rm (i)]
\item  
$\alpha$ is real  and $\delta (m_\alpha  )   = 1$, or

\item 
$\alpha$ is complex.
\end{enumerate}
\end{definition}

Vogan shows in \cite[Lem.~4.3.12]{Voganbook} that $\Delta_\delta(\mathfrak{l},\mathfrak{a}) $ is a root system, and  makes the following definition:

\begin{definition}[{\cite[Def.~4.3.13]{Voganbook}}]
\label{def-w-0-group}
The group  $W^0_\delta$ is  the subgroup of $W(\overline{\Delta}(\mathfrak{l},\mathfrak{a}) )$ generated by the reflections associated to good roots.
\end{definition}

%\begin{lemma}[{\cite[Lems.~4.3.12 and 4.3.14]{Voganbook}}]
%\label{semi weyl V}
%$W^0_\delta$ is   a normal subgroup.  \qed
%\end{lemma}

\begin{definition}[{\cite[Def.~4.3.13]{Voganbook}}]
Let $(\mathfrak{q},H,\delta)$ be a set of Vogan data. We define 
$R_\delta = W_\delta/W^0_\delta$. (This is actually a group since $W^0_\delta$ is   a normal subgroup of $W_\delta$ \cite[Lems.~4.3.12 and 4.3.14]{Voganbook}. But the group structure will not be used here.)
\end{definition}

\begin{theorem}[{\cite[Cor.~4.4.11 and Cor.~6.5.14]{Voganbook}}]
\label{thm-cohomological-induction-preserves-irreducibility}
The representation $X^L(\delta)$ of $L$ in Definition~\textup{\ref{def-vogan-representations}}  is a direct sum of $|R_\delta|$ inequivalent irreducible subrepresentations, and the same is true of the representation $X^G(\mathfrak{q}, H ,\delta)$ of $G$.
\end{theorem}

\begin{corollary}
\label{cor-r-delta-equals-r-sigma}
If  $(\kq, H, \delta)$ is a set of  Vogan data, and if    $[P, \sigma]$ is the corresponding associate class, as in Theorem \textup{\ref{thm-vogan-zuckerman1}}, then   
$|  R_\sigma| = | R_\delta | $. \qed
\end{corollary}

\begin{proof}
According to the theorem above, and what we have noted about the Knapp-Stein theory, both $|R_\delta|$ and $|R_\sigma|$ equal the number of irreducible summands of the representation $X^G(\mathfrak{q},H,\sigma)\cong \Ind_P^G \Hilbert_{\sigma}$. 
\end{proof}

%%%%%%%%%%%%%%%%%%%%%%%%%%%%%%%%%%%%%%%%%%%%%%%%%%%%%%%%%%%%%%%%%%%%%%%%%%%%%%%%%%%%
%%%%%%%%%%%%%%%%%%%%%%%%%%%%%%%%%%%%%%%%%%%%%%%%%%%%%%%%%%%%%%%%%%%%%%%%%%%%%%%%%%%%

\section{Essential Components of the Tempered Dual}
\label{sec-essential-components}

The following is Definition~4.1 in \cite{ConnesKasparovPaper1}. Its significance is that the essential components of the tempered dual of $G$   are precisely those that contribute to $C^*$-algebra $K$-theory; see \cite[Thm.~4.9]{ConnesKasparovPaper1}.

\begin{definition}
\label{def-associate-class-essential}
A component of the tempered dual of $G$ corresponding to the associate class $[P,\sigma]$ is \emph{essential} if the Knapp-Stein subgroup $W'_\sigma\triangleleft W_\sigma$ is trivial.
\end{definition}

\begin{lemma}
A $K$-conjugacy class $[\kq, H, \delta]$ corresponds to an essential component of the tempered dual, as in Section~\textup{\ref{sec-components-of-tempered-dual}}, if and only if  $W^0_\delta$ is the trivial subgroup of $W_\delta$.
\end{lemma}

\begin{proof}
Suppose that a component of the tempered dual corresponds to an associate class $[P,\sigma]$. It follows from Lemma \ref{lem-w-sigma-equals-w-delta} and Corollary~\ref{cor-r-delta-equals-r-sigma} that 
\[
| W'_\sigma| = 1 \quad \Leftrightarrow \quad |W^0_\delta| = 1
\]
The lemma now follows directly from the definition above.
\end{proof}

\begin{theorem}
\label{thm-essential-components-correspond-to-essential-data}
The component of the tempered dual that is labeled by the set of Vogan data $(\mathfrak{q},H,\delta)$ is essential if and only if $(\mathfrak{q},H,\delta)$ is  essential in the sense of Definition~\textup{\ref{def-essential-vogan-data}}.
\end{theorem}

This is a special case of   \cite[Lem.~4.3.31]{Voganbook}.  But since our result is somewhat hidden there, we shall give some of the details here.

\begin{proof}
If $(\mathfrak{q},H,\delta)$ is essential, then it is immediate from Definitions~\ref{def-essential-vogan-data}, \ref{def-good-roots} and \ref{def-w-0-group} that $W^0_\delta$ is trivial.
So suppose, conversely, that $W_\delta^0$ is trivial, so that, according to  Definitions~\ref{def-good-roots} and \ref{def-w-0-group}, 
 every root  in $\overline\Delta(\mathfrak{l},\mathfrak{a})$ is real, and moreover
 \begin{equation*}
 \delta (m_\alpha) \ne 1 \qquad \forall \alpha \in \overline \Delta(\mathfrak{l},\mathfrak{a}),
 \end{equation*}
 which implies that 
  \begin{equation}
     \label{eq-delta-of-m-is-minus-one}
 \delta (m_\alpha) = -1  \qquad \forall \alpha \in \overline \Delta(\mathfrak{l},\mathfrak{a}),
 \end{equation} 
 since $m_\alpha^2=e$. Write $L =N_G(\mathfrak{q})$ and  $H= TA$, as usual. 
 
 Our  first task is to show that the reduced root system $\overline\Delta(\mathfrak{l},\mathfrak{a})$ matches one of those obtained from  an essential $(\mathfrak{q},H,\delta)$. These are the systems  of type $A_1 \times \cdots \times A_1$. 

Now $\overline\Delta(\mathfrak{l},\mathfrak{a})$ is of type $A_1 \times \cdots \times A_1$ if and only if for   every   $\alpha, \beta \in \overline \Delta (\mathfrak{l}, \mathfrak{a})$ with  $\alpha \ne \pm \beta$ the roots in the plane spanned by $\alpha$ and $\beta$  span a system of type $A_1 \times A_1$.  But  apart from $A_1 \times A_1$, the only other possibilities are $A_2$, $B_2$ and $G_2$, and in all these cases there exist roots $\gamma_1$, $\gamma_2$, $\gamma_3$ for which the dual roots (which also constitute a root system of rank $2$) satisfy
\begin{equation}
    \label{eq-coroot-identity}
\check \gamma_1 = \check \gamma_2 + \check \gamma_3.
\end{equation}
But it is shown in \cite[Cor.~4.3.20]{Voganbook}, again by considering the various cases, that \eqref{eq-coroot-identity} implies
\[
\delta(m_{\gamma_1}) = \delta(m_{\gamma_2})\delta (m_{\gamma_3}),
\]
and this is obviously inconsistent with \eqref{eq-delta-of-m-is-minus-one}.
 
Now let $\alpha\in \overline \Delta(\mathfrak{l},\mathfrak{a})$. The root $\alpha$ is real, and is therefore  the restriction to $\mathfrak{a}\subseteq \mathfrak{h}$ of a root $\gamma\in \Delta (\mathfrak{l},\mathfrak{h})$ that vanishes on $\mathfrak{t}\subseteq \mathfrak{h}$. Standard arguments in Lie theory (see again \cite[Sec.~VI.2]{Serre} for instance) show that the spaces $\mathfrak{l}_\gamma $, $\mathfrak{l}_{-\gamma}$ and $[\mathfrak{l}_{\gamma}, \mathfrak{l}_{-\gamma}]$ are each one-dimensional and  span a copy of $\mathfrak{sl}(2,\C)\subseteq \mathfrak{l}$. The reality condition implies that $[\mathfrak{l}_{\gamma}, \mathfrak{l}_{-\gamma}]$ is included in $\mathfrak{a}\subseteq \mathfrak{h}$.  In fact if $H_\gamma\in \mathfrak{a}$ is chosen so that $\langle H,H_\gamma\rangle = \gamma (H)$ for all $H\in \mathfrak{h}$, then $[\mathfrak{l}_{\gamma}, \mathfrak{l}_{-\gamma}] = \C\cdot  H_\gamma$.

Now form the vector subspace   
\begin{equation}
    \label{eq-direct-sum-of-reduced-and-nonreduced-root-spaces}
\mathfrak{l}_{-2 \alpha} \oplus \mathfrak{l}_{- \alpha} \oplus [\mathfrak{l}_{\alpha},\mathfrak{l}_{- \alpha} ] \oplus \mathfrak{l}_{\alpha} \oplus \mathfrak{l}_{ 2\alpha} 
\end{equation}
within $\mathfrak{l}$. This includes our copy of $\mathfrak{sl}(2,\C)$ since $\mathfrak{l}_{\pm \gamma} \subseteq \mathfrak{l}_{\pm \alpha}$, and so the direct sum  is a module over $\mathfrak{sl}(2,\C)$. Because $H_\gamma \in \mathfrak{a}$ and because $\gamma \vert _{\mathfrak{a}} = \alpha$, the individual summands are the weight spaces for this module action. It therefore follows from the representation theory of $\mathfrak{sl}(2,\C)$ that if $X \in \mathfrak{l}_{\gamma}$ is nonzero, then the morphism
\[
\ad _X \colon \mathfrak{l}_{- \alpha} \longrightarrow  [\mathfrak{l}_{\alpha},\mathfrak{l}_{-\alpha} ]
\]
is injective.  But if $Y\in \mathfrak{l}_{-\alpha}$ and $H \in \mathfrak{h}$, then 
\[
\langle [X,Y], H\rangle = \langle Y, [H,X]\rangle = \langle Y,\gamma (H) X\rangle  
= \langle Y,X\rangle \langle  H_\gamma , H\rangle .
\]
As a result, $\ad_X (Y) \in \C\cdot  H_\gamma$, and therefore  $\operatorname{dim} ( \mathfrak{l}_{-\alpha}) {= }1$, which  means that $\mathfrak{l}_{-\alpha} = \mathfrak{l}_{-\gamma}$.  Similarly $\mathfrak{l}_{\alpha} = \mathfrak{l}_{\gamma}$. It  follows that the middle summand in \eqref{eq-direct-sum-of-reduced-and-nonreduced-root-spaces}, which is the weight zero space for the $\mathfrak{sl}(2,\C)$-representation, is one-dimensional, and so it follows from the representation theory of $\mathfrak{sl}(2,\C)$ that 
\[
\mathfrak{l}_{-2 \alpha} \oplus \mathfrak{l}_{- \alpha} \oplus [\mathfrak{l}_{\alpha},\mathfrak{l}_{- \alpha} ] \oplus \mathfrak{l}_{\alpha} \oplus \mathfrak{l}_{ 2\alpha} 
= \mathfrak{l}_{-\gamma} \oplus \C{\cdot}  H_\gamma \oplus \mathfrak{l}_{\gamma}.
\]
As a result,   the above is the complexification of $\phi_{\alpha}[\mathfrak{sl}(2,\R)]\subseteq \mathfrak{l}_0$.

Now if $\alpha\ne \pm \beta$, then the subalgebras $\phi_{\alpha}[\mathfrak{sl}(2,\R)]$ and   $\phi_{\beta}[\mathfrak{sl}(2,\R)]$ commute with one another and  are orthogonal to one another, so it follows from the above that   $\mathfrak{l}_0$ is a direct sum 
\[
\mathfrak{l}_0 
= 
\phi_{\alpha_1}[\mathfrak{sl}(2,\R)]  \oplus \cdots \oplus  \phi_{\alpha_N}[\mathfrak{sl}(2,\R)] \oplus \mathfrak{z_0} ,
\]
where the direct sum is over an enumeration of a system of positive roots, and this together with \eqref{eq-delta-of-m-is-minus-one} shows that $(\mathfrak{q},H,\delta)$ is essential.
\end{proof}

\section{Dirac Cohomology} 
\label{sec-dirac-cohomology}

In \cite[Sec.~5]{ConnesKasparovPaper1} we approached the Dirac operator  from a direction congenial to $K$-theory, but in representation theory  it is convenient to follow a different but equivalent approach, which we shall review here. For full details about the Dirac operator in representation theory, see the monograph \cite{HuangPandzic}.

\subsection*{The algebraic Dirac operator}
Pick an irreducible representation of associative algebras 
\[
\operatorname{Cliff}(\mathfrak{s}_0) \longrightarrow \End (\mathcal{S}).
\]
When $\dim(\mathfrak{s}_0)$ is even there is a unique irreducible representation, up to equivalence.  When $\dim(\mathfrak{s}_0)$ is odd there are two, but the main construction below (of  Dirac cohomology) is independent of the choice.

  \begin{definition}
\label{def-vogan-dirac-operator}
The \emph{Dirac operator} associated to an admissible (for instance irreducible)  unitary $(\mathfrak{g},K)$-module $X$ is the  linear operator
   \[ 
 \slashed{D}  \colon X \otimes   \mathcal{S} \longrightarrow  X \otimes \mathcal{S}   
 \]
 given by the formula 
 \[
 \slashed{D} = \sum X_a \otimes c(X_a),
 \]
 in which  the sum is over an orthonormal basis for $\mathfrak{s}_0$.
See  \cite{HuangPandzic}.  The \emph{Dirac cohomology} of $X$ is 
\[
\Dirac(X) = \operatorname{Kernel} \bigl (\slashed{D}\colon  X \otimes    \mathcal{S}  \longrightarrow  X \otimes  \mathcal{S}  \bigr  ).
\]
Occasionally we shall write $\Dirac^G(X)$, if it is helpful to emphasize the group involved.
 \end{definition}

The Dirac cohomology of $X$  is a finite-dimensional vector space.  It is also a representation of the \emph{spin double cover} of $K$, which is the 
  (not necessarily connected) double    cover $\tilde K$ of $K$ defined by the pullback diagram
\begin{equation}
\label{eq-spin-double-cover}
\xymatrix{
\tilde K \ar[r]\ar[d] &  \operatorname{Spin} (\mathfrak{s}_0)\ar[d]  \\
K\ar[r] & SO(\mathfrak{s}) . \\
}
\end{equation}
This follows from the existence of a natural Lie algebra homomorphism
\[
\mathfrak{k}_0 \longrightarrow \operatorname{Cliff} (\mathfrak{s}_0),
\]
see \cite[Sec.~2.3]{HuangPandzic}, which allows us to equip $X \otimes \mathcal{S}$ with a diagonal $\mathfrak{k}_0$-action that exponentiates to a $\tilde K$-action, with respect to which the Dirac operator is equivariant.

In \cite{ConnesKasparovPaper1}, in place of the irreducible $\operatorname{Cliff}(\mathfrak{s}_0)$-module $\mathcal{S}$,   we used  modules  of the form $ \mathcal{S}\otimes V_\pi^*$, where $\pi\colon \tilde K \to \Aut(V_\pi)$ is an irreducible genuine representation of $\tilde K$, and where by \emph{genuine} we mean that the representation is nontrivial on the kernel of $\tilde K \to K$. We formed the Dirac operator \[
\slashed{D}_{\mathcal{S}\otimes V_\pi^*}\colon [X \otimes \mathcal{S}\otimes V_\pi^*]^K \longrightarrow 
[X \otimes \mathcal{S}\otimes V_\pi^*]^K
\]
(the diagonal action of $\tilde K$ descends to $K$) and studied its kernel.  This approach is related to the definition of  Dirac cohomology by the formula
\begin{equation}
    \label{eq-relation-between-two-diracs}
\dim \bigl (\operatorname{kernel}   (   \slashed{D}_{\mathcal{S}\otimes V_\pi^*}   ) \bigr )
= 
\text{multiplicity of $\pi$ in  $\Dirac(X)$} .
\end{equation}
In what follows we shall use   \eqref{eq-relation-between-two-diracs}  to translate statements from the paper  \cite{ConnesKasparovPaper1} about $\operatorname{kernel} (\slashed{D}_{\mathcal{S}\otimes V_\pi^*})$ into statements about $\Dirac(X)$.
 
\subsection*{The matching theorem}
It follows from \cite[Thm.~5.27]{ConnesKasparovPaper1} that 
\[
\operatorname{kernel} (\slashed{D}_{\mathcal{S}\otimes V_\pi^*}) \ne 0
\quad \Leftrightarrow \quad [X \otimes \mathcal{S}\otimes V_\pi^*]^K \ne 0 \quad \& \quad \slashed{D}_{\mathcal{S}\otimes V_\pi^*} =0.
\]
The following is therefore the translation, using \eqref{eq-relation-between-two-diracs} and Theorem~\ref{thm-vogan-zuckerman1}, of Definition~6.1 in \cite{ConnesKasparovPaper1}. 
 
  \begin{definition}
A genuine, irreducible representation $\pi$ of $\tilde K$  and a component $[\mathfrak{q},H,\delta]$ of the tempered dual are \emph{matched} if $\pi$ occurs in the Dirac cohomology space   $\Dirac (X^G(\mathfrak{q}, H,\delta))$.
 \end{definition}

The following is the translation of   Theorem~6.3 in \cite{ConnesKasparovPaper1}:

\begin{theorem}[Matching Theorem]
\label{thm-matching-theorem}
Let $G$ be a connected linear real reductive group. Let $K$ be a maximal compact subgroup of $G$, and let $\tilde K$ be its spin double cover, as in~\eqref{eq-spin-double-cover}.
\begin{enumerate}[\rm (i)]

\item For every essential component of the tempered dual of $G$ there is a unique genuine, irreducible representation of $\tilde K$ to which it is matched.

\item For every genuine, irreducible representation of $\tilde K$  there is a unique essential component of the tempered dual of $G$ which it is matched.
\end{enumerate}
\end{theorem}

As for the  genuine, irreducible representation of $\tilde K$, they may be classified by highest weight theory, as follows.  Pick any system of positive roots $\Delta^+(\mathfrak{g}, \mathfrak{t}^c)$ for the action of $\mathfrak{t}^c$ on $\mathfrak{g}$.  We can do so in a way that includes our choice of $\Delta^+ (\mathfrak{k},\mathfrak{t}^c)$, and then there is a partition 
\[
\Delta^+(\mathfrak{g}, \mathfrak{t}^c) =
\Delta^+(\mathfrak{k}, \mathfrak{t}^c) \cup 
\Delta^+(\mathfrak{s}, \mathfrak{t}^c)
\]
into compact and noncompact positive roots. The weight $\rho (\Delta^+(\mathfrak{s}, \mathfrak{t}^c))\in \mathfrak{t}_0^{c^*}$ is the highest weight of the spin representation $\mathcal{S}$, and because of this we have: 

\begin{proposition}
\label{prop-genuine-reps-of-k-tilde-by-highest-weight}
Each genuine, irreducible representation of $\tilde K$ has a unique highest weight $\kappa\in i \mathfrak{t}^{c*}_0$.  The weight 
\[
\mu = \kappa - \rho (\Delta^+(\mathfrak{s}, \mathfrak{t}^c)) \in i \mathfrak{t}_0^{c^*}
\]
is analytically integral for $T^c$.  Conversely, if $\kappa\in i \mathfrak{t}_0^{c*}$ is a dominant weight, and if the weight $\mu$ above is integral, then $\kappa$ is the highest weight of a unique genuine irreducible representation of $\tilde K$. \qed
\end{proposition}

As usual, the integrality condition in the proposition is independent of the choice of positive system $ \Delta^+ (\mathfrak{g},\mathfrak{t}^c)$.  But   given a dominant weight $\kappa\in i \mathfrak{t}^{c*}_0$, we could form the $\theta$-stable parabolic subalgebra $\mathfrak{q} = \mathfrak{l} + \mathfrak{u}$ in Proposition~\ref{prop-theta-stable-data-from-dominant-weight}, then choose a positive system $\Delta^+( \mathfrak{l},\mathfrak{t}^c)$ as in Proposition~\ref{prop-theta-stable-data-from-dominant-weight}, and then define
\[
\rho(\Delta ^+ (\mathfrak{s}, \mathfrak{t}^c)) = 
\rho (\mathfrak{s} \cap \mathfrak{u}) + \rho ( \Delta ^+(\mathfrak{l},\mathfrak{t}^c)).
\]
Compare Remark~\ref{rem-weights-from-u-and-l-give-positive-system-for-g}. 
It follows that  the integrality condition in Proposition~\ref{prop-genuine-reps-of-k-tilde-by-highest-weight} above is exactly the same as the integrality condition in Proposition~\ref{prop-theta-stable-data-from-dominant-weight}.
Therefore, thanks to Theorem~\ref{thm-classification-of-essential-discrete-data}, this means that: 

\begin{theorem}
The correspondence that associates 
to a genuine representation of $\tilde K$  with highest weight $\kappa$   the set of Vogan data generated by $\kappa$, as in Proposition~\textup{\ref{prop-theta-stable-data-from-dominant-weight}},
determines  a bijection from the equivalence classes of genuine irreducible representations of $\tilde K$ to the $K$-conjugacy classes of essential sets of Vogan data that maps. \qed
\end{theorem}

It only remains to describe this bijection in terms of Dirac cohomology.
%%%%%%%%%%%%%%%%%%%%%%%%%%%%%%%%%%%%%%%%%%%%%%%%%%%%%%%%%%%%%%%%%%%%%%%%%%%%%%%%%%%%
%%%%%%%%%%%%%%%%%%%%%%%%%%%%%%%%%%%%%%%%%%%%%%%%%%%%%%%%%%%%%%%%%%%%%%%%%%%%%%%%%%%%

%%%%%%%%%%%%%%%%%%%%%%%%%%%%%%%%%%%%%%%%%%%%%%%%%%%%%%%%%%%%%%%%%%%%%%%%%%%%%%%%%%%%
%%%%%%%%%%%%%%%%%%%%%%%%%%%%%%%%%%%%%%%%%%%%%%%%%%%%%%%%%%%%%%%%%%%%%%%%%%%%%%%%%%%%

\subsection*{Dirac cohomology for essential theta-stable data}
Let $(\mathfrak{q},H,\delta)$ be a set of essential Vogan data, let $L=N_G (\mathfrak{q})$ and let $H=TA$.
Write 
\[
\kl_0 \cong  \underbrace{\mathfrak{sl}(2,\R)   \oplus \cdots \oplus   \mathfrak{sl}(2,\R) }_{\text{$N$ times}}  \oplus \mathfrak{z}_0
\]
where $\mathfrak{z}_0$ is the center of $\mathfrak{l}_0$.

\begin{lemma}
\label{lem-dirac-cohonmology-for-l}
  Let $\kappa ^L\in i \mathfrak{t}_0^{c*}$ be the extension by $0$ on the orthogonal complement of $\mathfrak{t}_0$ of the differential of $\delta$. The Dirac cohomology space 
$\Dirac^L (X (\delta))$ is a direct sum of $2^N$ copies of the  genuine  unitary character of the spin-cover of  $  T^c$ with differential $\kappa^L$.
\end{lemma}
 
\begin{proof}
This is an explicit computation using the odd principal series for $SL(2,\R)$.  The only subtle point is to construct enough vectors in the principal series representation $X(\delta)$ to account for $\Dirac^L(X(\delta ))$.  These vectors are functions on $T^c$ (which is the $\theta$-fixed maximal compact subgroup of $L$) that transform on the right according to the character $\delta : T\to U(1)$, and the particular functions needed are those that are supplied by Lemma~\ref{lem-fine-k-type-for-special-l}.
\end{proof}

\subsection*{Dirac cohomology and cohomological induction}

As explained in Section~\ref{sec-components-of-tempered-dual}, the representations $X(\mathfrak{q},H , \delta)$ are obtained from the representations $X(\delta)$ by means of cohomological induction.  Fortunately Dirac cohomology and cohomological induction are very easily related to one another, as follows:

\begin{theorem}
\label{thm-dirac-cohomology-and-cohomological-induction}
Let $G$ be a linear, connected, real reductive group with maximal compact subgroup $K$ and let $(\mathfrak{q},H,\delta)$ be a set of essential Vogan data for $G$.   Write $\mathfrak{q} = \mathfrak{l} \oplus \mathfrak{u}$ and $\mathfrak{h}_0 = \mathfrak{t}_0 \oplus \mathfrak{a}_0$, and let $\kappa^L  \in i \mathfrak{t}_0^{c*}$ be the extension by zero of the differential of $\delta$.   Let 
\[
\kappa^G = \kappa^L + \rho ( \mathfrak{u} \cap \mathfrak{s}, \mathfrak{t}^c).
\]
The Dirac cohomology space  $\Dirac^G(X(\mathfrak{q},H,\delta))$ is a direct sum of $2^N$ copies  the irreducible, genuine representation of $\tilde K$ with highest weight $\kappa^G$.  
\end{theorem} 

\begin{proof}
There is a general formula in \cite{DongHuang15} for the Dirac cohomology of a cohomologically induced representation in terms of the Dirac cohomology of the initial representation  that applies when the  infinitesimal character of the initial representation is  is \emph{weakly good} relative to the $\theta$-stable parabolic subalgebra   used in the induction construction.   We shall explain how this formula proves the theorem.

In our case, the infinitesimal character  of the initial representation   $X^L(\delta)$   is   $\lambda^L$,   the differential of $\delta$; see \cite[Lem.~4.1.8]{Voganbook}. Here $\lambda^L$ is of course a weight of $\mathfrak{t}$, and   to identify it with the infinitesimal character we extend $\lambda^L$ by zero on $\mathfrak{a}$ so that it becomes a weight of the Cartan subalgebra $\mathfrak{h} = \mathfrak{t} {+} \mathfrak{a}$. 

The weakly good condition for a general infinitesimal character $\Lambda$ is   that 
\[
\operatorname{Re} \bigl ( \langle  \Lambda + \rho(\mathfrak{u},\mathfrak{t}) , \alpha 
\rangle \bigr )  \ge 0\qquad \forall\, \alpha\in \Delta (\mathfrak{u}, \mathfrak{t}).
\]
This holds in our case   since $\lambda^L + \rho (\mathfrak{u},\mathfrak{t})=\lambda ^G$; see part (v) of Definition~\ref{def-theta-stable-data}.

To explain what the formula says, denote by $\mathcal{L}^G_S$ the functor of cohomological induction in degree $S=\dim (\mathfrak{u}\cap \mathfrak{s})$, and along $\mathfrak{q}$, from $(\mathfrak{l},T^c)$-modules to   $(\mathfrak{g},K)$ modules. In addition denote by $\mathcal{L}^{\widetilde K}_S$
the functor of cohomological induction in the same degree,
along $\mathfrak{u} \cap \mathfrak{k}$, from  $\overline T^c$-modules to   $\tilde K$-modules, where $\overline T_c $ is the subgroup of $\widetilde K$ that covers $T^c$ (this is not necessarily the same thing as the spin double cover of $T^c$).  See the introductory sections in \cite{DongHuang15} for details. It is explained in \cite[Lem.~5.1]{DongHuang15} how the tensor product 
\[
\Dirac^L(X(\delta))' : = \Dirac^L(X(\delta))\otimes \C_{-\rho (\mathfrak{u}\cap \mathfrak{s}, \mathfrak{t}^c)}
\]
carries the structure of a $\overline T^c$-module, and then proved in \cite[Thm.~5.7]{DongHuang15} that 
\[
\Dirac^G \bigl ( \mathcal{L}_S^{G}  (X )  \bigr ) 
\cong 
\mathcal{L}_S^{\tilde K}\bigl ( \Dirac^L(X(\delta))' \bigr ) .
\]
But it follows from our Lemma~\ref{lem-dirac-cohonmology-for-l} that 
\[
\Dirac^L(X(\delta))' = \C_{\kappa_G- 2 \rho (\mathfrak{u}\cap \mathfrak{s}, \mathfrak{t^c})}
\oplus \cdots \oplus \C_{\kappa_G- 2 \rho (\mathfrak{u}\cap \mathfrak{s}, \mathfrak{t^c})}
\]
($2^N$ copies), while according to \cite[Prop.~4.1]{DongHuang15}
or \cite[Cor.~5.72]{KnappVoganBook},
\[
\mathcal{L}_S^{\tilde K}\bigl ( \C_{\kappa_G- 2 \rho (\mathfrak{u}\cap \mathfrak{s}, \mathfrak{t^c})} \bigr ) = \text{genuine rep.~of $\tilde K$ with highest weight $\kappa_G$.}
\]
The result follows.
\end{proof}
 
\subsection*{Proof of the matching theorem} 
The proof of the matching theorem is now within easy reach: 

\begin{proof}[Proof of Theorem~\ref{thm-matching-theorem}]
Lemma~\ref{lem-dirac-cohonmology-for-l} and Theorem~\ref{thm-dirac-cohomology-and-cohomological-induction} imply that if the set of essential Vogan data $(\mathfrak{q},H,\delta)$ is constructed from the dominant weight $\kappa$, as in Proposition~\ref{prop-theta-stable-data-from-dominant-weight}, then the Dirac cohomology of $X(\mathfrak{q},H,\delta)$ is a nonzero multiple of the irreducible representation of $\tilde K$ with highest weight $\kappa$.  The matching theorem is therefore a consequence of Theorems~\ref{thm-classification-of-essential-discrete-data},  \ref{thm-vogan-zuckerman0} and \ref{thm-essential-components-correspond-to-essential-data}, and the classification of the irreducible representations of $\tilde K$ by highest weight. 
\end{proof}

 %%%%%%%%%%%%%%%%%%%%%%%%%%%%%%%%%%%
 %%%%%%%%%%%%%%%%%%%%%%%%%%%%%%%%%%%
 %%%%%%%%%%%%%%%%%%%%%%%%%%%%%%%%%%%
 %%%%%%%%%%%%%%%%%%%%%%%%%%%%%%%%%%%

 \section{More on Dirac Cohomology and Minimal K-Types}
\label{sec-more-on-dirac-cohomology}

In this concluding section we shall make three points that are related to some small additional requirements of the proof of the Connes-Kasparov conjecture in \cite{ConnesKasparovPaper1}, and to the problem of computing the Connes-Kasp\-arov isomorphism.  

First we shall make note of a detail from Section~\ref{sec-intertwining-groups} of this paper that is used in the $C^*$-algebra $K$-theory computation in \cite{ConnesKasparovPaper1}; see Theorem~3.7 there. 

\begin{theorem}
\label{thm-size-of-r-group}
If $(\mathfrak{q}, H, \delta)$ is a set of essential Vogan data, and if $H=TA$, then
$|R_\delta | = 2^N$, where $N= \dim(\mathfrak{a}) - \operatorname{rank}(G)+ \operatorname{rank}(K)$.
\end{theorem}

\begin{proof}
Since $H$ is a Cartan subgroup of $G$, 
 \[
 \operatorname{rank} (G) = \dim (\mathfrak{t} ) + \dim (\mathfrak{a}),
 \]
 and since $T^c$ is a maximal torus in $K$, 
 \[
  \operatorname{rank} (K) = \dim (\mathfrak{t}^c ).
 \]
 If $(\mathfrak{q},H,\delta) $ is essential, then $|R_\delta| = 2 ^N$, where $N$ is the number of $\mathfrak{sl}(2,\R)$ summands in the direct sum decomposition of  $\mathfrak{l}_0$ appearing in Definition~\ref{def-essential-vogan-data}.  But this number $N$ is also the difference in dimension between $\mathfrak{t}^c $ and $\mathfrak{t}$, and so 
 \[
N = \dim (\mathfrak{t}^c ) - \dim(\mathfrak{t}).
 \]
 The formula in the statement of the theorem follows from the three displayed identities.
\end{proof} 
 
Secondly, we shall prove a theorem that describes the Dirac cohomology of an essential component of the tempered dual in a bit more detail.  The theorem was used in the second proof of the Connes-Kasparov isomorphism in \cite{ConnesKasparovPaper1}, which, in comparison to the first proof,   uses less $K$-theory but more representation theory. However the proof was deferred until now.  Here is the result,  translated into the language of the present paper; it was stated as  Theorem~8.4 in \cite{ConnesKasparovPaper1}. 

\begin{theorem}
\label{thm-minimal-k-types}
Let $(\mathfrak{q},H,\delta)$ be a  set of essential Vogan  data. 
\begin{enumerate}[\rm (i)]

\item 
Each minimal $K$-type of $X(\mathfrak{q},H,\delta)$ has multiplicity one, and each irreducible direct summand   of $X(\mathfrak{q},H,\delta)$  includes precisely one of these minimal $K$-types.

\item   If   $\tau$ is the irreducible representation of $\tilde K$ to which $X(\mathfrak{q},H,\delta)$ is matched,
then the $\tilde K$-module $\Dirac(X(\mathfrak{q},H,\delta))$ is the direct sum of 
$ |R_\delta| $ copies of $\tau$.

\item If $X$ is an irreducible summand of $X(\mathfrak{q},H,\delta)$ and if $V \subseteq X  $ is its minimal $K$-type, and if $[V  \otimes \mathcal{S} ]^\tau$ is the $\tau$-isotypical summand in the tensor product, then  
\[
[V  \otimes \mathcal{S}]^\tau \subseteq \Dirac (X)
\]
and the inclusion  is a vector space isomorphism.

\end{enumerate}
\end{theorem}

\begin{proof}
Statement (i) is contained in \cite[Thm.~6.5.9]{Voganbook}. 
Statement (ii) is Theorem~\ref{thm-dirac-cohomology-and-cohomological-induction}.
Statement (iii) is proved in  \cite[Sec.~4]{DingDong16} (see the final paragraph of the proof of Theorem~1.2 there; our statement (iii) appears in the proof, but not in the formulation, of that theorem).
\end{proof} 

\begin{remark} 
The theorem implies that if $(\mathfrak{q},H,\delta)$ is essential, then all of the irreducible summands of $X(\mathfrak{q},H,\delta)$ have isomorphic Dirac cohomology spaces, consisting of a single copy of the representation $\tau$.  It follows from \cite[Thm.~7.5]{DongHuang15} and the results of this paper that if $(\mathfrak{q},H,\delta)$ is not essential, then the Dirac cohomology of $X(\mathfrak{q},H,\delta)$ is zero.
% If $(\mathfrak{q},H,\delta)$ is not essential, then the Dirac cohomology is zero.
\end{remark}

Our third point has to do with computation of  the Connes-Kasparov isomorphism in examples.  We described  the result of the computation for  $Sp(4,\R)$ in \cite[Sec.~6]{ConnesKasparovPaper1}.  Here we shall  explain how to carry out those  computations   using minimal $K$-types. 

\begin{figure}[ht]
  \[ 
    \renewcommand{\arraystretch}{2.1}
\begin{array}{ccccccccc}
    & &&&&&&& \WBullet\\
    & &&&&&&\WBullet  &\WBullet \\
    & &&&&& \sigma_{2,0,1} & \sigma_{2,0,2}& \sigma_{2,0,3} \\
    & &&&&\sigma_{\min,2}&\sigma_{2,1,1}&\sigma_{2,1,2}&\sigma_{2,1,3} \\
    & &&&\sigma_{\min,0}&\sigma_{\min,1}&\sigma_{2,0,1}&\sigma_{2,0,2}&\sigma_{2,0,3} \\
    & &&\sigma_{\min,2}&\sigma_{\min,1}&\sigma_{1,1}&\sigma_{1,2}&\WBullet&\WBullet \\
    &         &   \sigma_{2,0,-1}      &\sigma_{2,1,-1} &\sigma_{2,0,-1}&\sigma_{1,2} &\sigma_{1,3}&\sigma_{1,4} &\WBullet \\
  &   \WBullet        &\sigma_{2,0,-2}  &\sigma_{2,1,-2} &\sigma_{2,0,-2}&\WBullet &\sigma_{1,4}&\sigma_{1,5}&\sigma_{1,6}  \\
    \WBullet&\WBullet  &\sigma_{2,0,-3}  &\sigma_{2,1,-3} &\sigma_{2,0,-3}&\WBullet &\WBullet & \sigma_{1,6} &\sigma_{1,7} \\
\end{array} 
   \]
\caption{
The nodes in this diagram are the integer lattice points $(m,n)$  with $m{\ge} n$, with the entry $\sigma_{\min,0}$  in position $(0,0)$. The pairs $(m,n)$  label the highest weights of the irreducible representations of the maximal compact subgroup  $K{\cong} U(2)$ in $Sp(4,\R)$ in the usual way.   The bullet points represent the discrete series representations of $Sp(4,\R)$, and a bullet point at $(m,n)$ indicates that $(m,n)$ occurs as a minimal $K$-type of a discrete series representation. The other labels represent  discrete representations of (the compactly generated parts of)  Levi subgroups of $Sp(4,\R)$, as  described in \cite[Ex.~6.7]{ConnesKasparovPaper1}. A label is placed in position $(m,n)$ if the corresponding parabolically induced representation of $Sp(4,\R)$ includes the $U(2)$-representation with highest weight $(m,n)$ as a minimal $K$-type.  
}
\label{fig-sp4-minimal-k-type-summary}
\end{figure}

The group $G=Sp(4,\R)$ has four associate classes of parabolic subgroup, including $G$ itself, and they are all cuspidal, meaning that they all contribute components to the tempered dual.  In the case of $P{=}G$, the minimal $K$-types of the representations in the associated components, namely the discrete series of $G$, are readily obtained from a well-known universal formula, valid for all reductive  groups with discrete series (see for instance \cite[Thm.~8.5]{AtiyahSchmid}, although from the point of view taken in the present paper it would perhaps be more appropriate to extract this from \cite[Thm.~6.5.9]{Voganbook}).  The  minimal $K$-types of representations in   the components associated to the other parabolics  can be  computed by hand, and the results are shown in Figure~\ref{fig-sp4-minimal-k-type-summary}.

Using this information, together with part (i) of Theorem~\ref{thm-minimal-k-types} and a computation of the groups $W_\sigma$, the set of essential components may be readily determined.  For instance the representation parabolically induced from  $\sigma_{1,2}$ has two minimal $K$-types, while one may   compute that $W_{\sigma_{1,2}}\cong \Z_2$, and  so according to the discussion in Sections~\ref{sec-intertwining-groups} and \ref{sec-essential-components}, the associated component of the tempered dual is essential.    The full set of all essential components is displayed in \cite[Fig.~1]{ConnesKasparovPaper1}.

The following result explains, in terms of minimal $K$-types, which essential components are matched to which genuine irreducible representations of $\widetilde K$.

\begin{theorem}
Let $\pi$ be a tempered, irreducible representation of $G$ that lies in an essential component of the tempered dual of $G$, and let $\mu$ be the highest weight of a minimal $K$-type of $\pi$.  If $\Delta^+(\mathfrak{g}, \mathfrak{t}^c)$  is a system of positive weights for $(\mathfrak{g},\mathfrak{t}^c)$ with respect to which $\mu + 2 \rho _K$ is dominant, and if $\rho_G$ is the associated half-sum of positive roots, then the essential component to which $\pi$ belongs is matched to the genuine irreducible representation of $\widetilde K$ with highest weight $\kappa = \mu -\rho_G +  \rho _K$.
\end{theorem}

\begin{remark}
    It follows from Proposition~\ref{prop-genuine-reps-of-k-tilde-by-highest-weight} that  the weight $\kappa$ above is indeed the highest weight of \emph{some} genuine irreducible representation of $\widetilde K$.
\end{remark}

\begin{proof}
The hypotheses imply that $\mu$ is the highest weight of a minimal $K$-type of some $X(\mathfrak{q},H,\delta)$, where  $(\mathfrak{q},H,\delta)$ is essential and where $\mathfrak{q}$ is defined by a strictly $\mathfrak{k}$-dominant weight.  To begin the proof of the theorem we shall  appeal to a result of Vogan and Zuckerman \cite[Thm.~6.5.9]{Voganbook} that describes the highest weights of the minimal $K$-types of $X(\mathfrak{q},H,\delta)$ as follows.

Form, as in Lemma~\ref{lem-fine-k-type-for-special-l}  and  Remark~\ref{rem-fine-weigths-of-tc-1},  the set $A(\mathfrak{q},\delta)$ of all fine  characters $\exp(\mu^L)$ of $T^c$  that   restrict to $\delta$ on $T$. For each of the weights $\mu^L$, form the sum
 \[
 \mu^G = \mu^L + 2 \rho (\mathfrak{s}\cap \mathfrak{u}).
 \]
 According to Vogan and Zuckerman,  the   weights $\mu^G$ so-obtained are precisely the highest weights of the minimal $K$-types of $X(\mathfrak{q},H,\delta)$.

It therefore follows from Remark~\ref{rem-fine-weigths-of-tc-2} that the  highest weights of the minimal $K$-types are all of the form
\[
\mu^G = \kappa^L  + \rho(\Delta^{+}(\mathfrak{l}, \mathfrak{t}^c))+  2 \rho (\mathfrak{u}\cap \mathfrak{s}) ,
\]
where $\Delta^{+}(\mathfrak{l}, \mathfrak{t}^c)$ is a system of positive roots for $(\mathfrak{l},\mathfrak{t}^c)$ and where $\kappa^L$ is the extension of the differential of $\delta\in \widehat T$ by $0$ to a weight in $i \mathfrak{t}_0^*$.

Theorem~\ref{thm-dirac-cohomology-and-cohomological-induction} asserts that  the Dirac cohomology space  $\Dirac^G(X(\mathfrak{q},H,\delta))$ is a non-zero direct sum of copies  of the irreducible, genuine representation of $\tilde K$ with highest weight 
 \begin{equation*}
     % \label{eq-kappa-g-in-terms-of-kappa-l}
 \kappa^G = \kappa^L + \rho (\mathfrak{u}\cap \mathfrak{s}, \mathfrak{t}^c).
 \end{equation*}
 So the highest weight $\mu^G$ of any minimal $K$-type and the highest weight of the matching  irreducible, genuine representation of $\tilde K$  are related by the formula 
 \begin{equation}
     \label{eq-kappa-g-in-terms-of-mu-g}
 \kappa^G = \mu^G + \rho(\Delta^{+}(\mathfrak{l}, \mathfrak{t}^c))+   \rho (\mathfrak{u}\cap \mathfrak{s}).
 \end{equation}
 
 Now  the union of $\Delta ^{+}(\mathfrak{l},\mathfrak{t}^c)$ with the given, fixed system of positive roots $\Delta^+(\mathfrak{k},\mathfrak{t}^c)$ and   $\Delta(\mathfrak{u}\cap \mathfrak{s})$  is a positive system for 
 $\Delta(\mathfrak{g},\mathfrak{t})$ for which the  half-sum of the positive roots is
 \begin{equation}
     \label{eq-rhos-formula}
     \rho_G
     =  \rho(\Delta^{+}(\mathfrak{l}, \mathfrak{t}^c))+   \rho (\mathfrak{u}\cap \mathfrak{s}) + \rho_K,
 \end{equation}
 with $\rho_K = \rho (\Delta^+(\mathfrak{k},\mathfrak{t}^c))$. We shall show  that $\mu^G + 2\rho_K$ is strictly dominant for this positive system, which will imply that this is the \emph{only} positive system with respect to which $\mu^G$ is dominant.  This, together    with the formula 
 \[
  \kappa^G = \mu^G + \rho_G -  \rho_K ,
  \]
which is a consequence of \eqref{eq-kappa-g-in-terms-of-mu-g} and \eqref{eq-rhos-formula}, will imply the theorem.
  
Strict dominance means that 
\[
\bigl \langle \mu^G + 2\rho_K, \alpha \bigr \rangle > 0 \qquad \forall \alpha \in \Delta ^+(\mathfrak{g}, \mathfrak{t}^c).
\]
To prove these inequalities, note first that   $\rho _K =\rho (\mathfrak{u} \cap \mathfrak{k}, \mathfrak{t}^c) $, since $\mathfrak{q}$ is defined by a strictly  dominant weight. This implies that   
 \[
 \mu^G + 2 \rho _K = \bigl ( \kappa^L + \rho (\mathfrak{u},\mathfrak{t}^c )\bigr) + \rho_G .
 \]
The first term on the right-hand side of the above equality is $\Delta^+(\mathfrak{g},\mathfrak{t}^c)$-dominant, because
\[ 
\alpha \in \Delta (\mathfrak{u},\mathfrak{t}^c) \quad 
\Rightarrow \quad 
\langle \kappa^L + \rho (\mathfrak{u},\mathfrak{t}^c ), \alpha \rangle > 0
\]
by part (v) of the definition of Vogan data (Definition~\ref{def-theta-stable-data}) and Lemma~\ref{lem-restriction-to-t-from-tc}. In addition    
\[ 
\alpha \in \Delta (\mathfrak{l},\mathfrak{t}^c) \quad 
\Rightarrow \quad  \langle  \kappa^L, \alpha \rangle = 0
 \quad \text{and} \quad 
\langle  \rho (\mathfrak{u},\mathfrak{t}^c ), \alpha \rangle = 0,
\]
since $\mu^L$ is an extension by zero in the first case, and by basic result about  parabolic subalgebras in the second \cite[Cor.~5.100]{KnappBeyond}.  So, after adding $\rho_G$, we find that $\mu^G + 2 \rho _K$ is  a strictly dominant $\Delta^+(\mathfrak{g},\mathfrak{t}^c)$-weight, as required.
\end{proof}

\subsection*{Acknowledgements}
The authors are very grateful to the referees for their numerous helpful suggestions.
This research was supported by NSF grants DMS-1952669 (NH),   DMS-1800667 and DMS-1952557 (YS).
Part of the research   was carried out within the online Research Community on Representation Theory and Noncommutative Geometry sponsored by the American Institute of Mathematics.

\bibliographystyle{alpha}
\bibliography{biblio}

\end{document}